\documentclass[12pt,reqno,hidelinks]{amsart}
\usepackage{amssymb}
\usepackage{amsmath, amssymb, verbatim, url, mathrsfs}
\usepackage{mathrsfs}
\usepackage{mathtools}
\usepackage{verbatim}
\usepackage{amsthm}
\usepackage{framed}
\usepackage{cite}
\usepackage{wasysym}
\usepackage{upgreek}
\usepackage{color}
\usepackage[dvipsnames]{xcolor}
\usepackage{tensor}
\usepackage{accents}
\usepackage{dsfont}
\usepackage{hyperref}
\usepackage{enumerate}
\usepackage[normalem]{ulem}
\usepackage{longtable}
\usepackage{mathtools}

\numberwithin{equation}{section}

\newtheorem{proposition}{Proposition}[section]
\newtheorem{lemma}[proposition]{Lemma}
\newtheorem{corollary}[proposition]{Corollary}
\newtheorem{theorem}[proposition]{Theorem}

\theoremstyle{definition}
\newtheorem{definition}[proposition]{Definition}
\newtheorem{remark}[proposition]{Remark}

\newcommand{\vertiiii}[1]{{\left\vert\kern-0.25ex\left\vert\kern-0.25ex\left\vert\kern-0.25ex\left\vert #1 \right\vert\kern-0.25ex\right\vert\kern-0.25ex\right\vert\kern-0.25ex\right\vert}}
\newcommand{\vertiii}[1]{{\left\vert\kern-0.25ex\left\vert\kern-0.25ex\left\vert #1 \right\vert\kern-0.25ex\right\vert\kern-0.25ex\right\vert}}

\newcommand{\ks}{\mathfrak{s}}
\newcommand{\kss}{\mathfrak{S}}
\newcommand{\ke}{\mathfrak{e}}
\newcommand{\kb}{\mathfrak{b}}

\newcommand{\rrho}{\rho}

\newcommand{\rw}{\mathring{w}}

\newcommand{\rp}{p}
\newcommand{\rPhi}{\Phi}

\newcommand{\bx}{\mathbf{x}}
\newcommand{\by}{\mathbf{y}}

\newcommand{\br}{\mathbf{r}}

\newcommand{\uW}{\underline{W}}
\newcommand{\uTheta}{\underline{\Theta}}
\newcommand{\uXi}{\underline{\Xi}}
\newcommand{\uOmega}{\underline{\Omega}}

\newcommand{\Rbb}{\mathbb{R}}
\newcommand{\Zbb}{\mathbb{Z}}

\newcommand{\dq}{\tilde{q}}

\newcommand{\hU}{\underaccent{\tilde}{U}}
\newcommand{\hV}{\underaccent{\tilde}{V}}

\newcommand{\hy}{\tilde{Y}}

\newcommand{\del}[1]{{\partial_{#1}}}

\newcommand{\AND}{{\quad\text{and}\quad}}

\newcommand{\Li}{L^\infty}

\newcommand{\starcup}{$\sqcup$\kern-0.58em{$\star$}}

\newcommand{\longeq}{\scalebox{3}[1]{=}}

\let\origmaketitle\maketitle
\def\maketitle{
	\begingroup
	\def\uppercasenonmath##1{} 
	\let\MakeUppercase\relax 
	\origmaketitle
	\endgroup
}

\allowdisplaybreaks

\DeclareMathOperator{\supp}{supp}

\begin{document}

\title[Blowups and longtime developments of irregularly-shaped molecular clouds]{Blowups and longtime developments with near-boundary mass accretions of irregularly-shaped Euler--Poisson dominated molecular clouds in astrophysics}

\address{Center for Mathematical Sciences, Huazhong University of Science and Technology,
1037 Luoyu Road, Wuhan, Hubei Province, China; Beijing International Center for Mathematical Research (BICMR), Peking University, No.5 Yiheyuan Road Haidian District, Beijing, China. }
\email{chao.liu.math@foxmail.com}
\author{Chao Liu}

\begin{abstract}
Motivated by the astrophysical problems of star formations from molecular clouds, we make the first step on the possible long time behaviors of certain irregularly-shaped molecular clouds.
We emphasis the \textit{main difficulty} of the blowups of the irregular-shaped fluids  with vacuum (molecular clouds) comes from the initial irregular configurations of its density (multiple centers of gravity). This inevitably causes more complicated movements during the evolution than the one with spherical symmetry. The spherical symmetric case has been well studied. However, for the non-spherical symmetric case with the gravity, it is very rare in the references due to a very complicate nonlinear interaction between the gravity and the fluids.
This article concludes, under the \textit{admissible data} (i.e., large scale, irregularly-shaped, expanding and rotational molecular clouds), the developments of the solution (molecular clouds) are either global (the first class) with near-boundary
mass accretions (lead to star formations), or blowup at finite time.
In addition, certain singularities can be removed from the boundary if the data is strongly admissible. This paper partially answers Makino's conjecture \cite{Makino1992} in $1992$ on the finite time blowup of any tame solution without symmetries for some data. The model of the molecular clouds and the local wellposedness have been established in the companion article \cite{Liu2021d}. 

\vspace{2mm}


\vspace{2mm}

\end{abstract}

\maketitle


\section{Introduction} \label{S:INTRO}

Astrophysics, both observational and theoretical (see, e.g., \cite{Larsona,LeBlanc2010,Ward-Thompson2011}), indicates that stars are formed or ``born'' in a truly \textit{vast cold  interstellar gas clouds} (or the \textit{giant molecular clouds}). Intuitively, since the molecular clouds are enormous, usually
containing up to a million solar masses of gas, the self-gravity is huge and may cause the interstellar
gas to condense into new stars. Therefore, \textit{mathematically} it is can be simplified as a special blowup problem (mass accretions) of the density of perfect fluids (molecular clouds) with vacuum under Newtonian self-gravity. The model of the molecular clouds and the local wellposedness have been established in the companion article \cite{Liu2021d}, and this paper \textit{aims} to study the \textit{possible blowups} and find out the \textit{mass accretion due to the free-fall boundary} (or called diffuse boundary since it models the diffuse clouds). The mass accretions due to the interior structures of the self-gravitational fluids (i.e. the linearized and nonlinear Jeans instability) will be given in our proceeding papers based on different methods (see our recent paper \cite{Liu2022} for a slightly nonlinear Jeans instability).

Standard astrophysical theory suggests there are several famous stages for the formations of stars from the interstellar gas clouds although the detailed processes are very model dependent. Most theoretical results are based on very ideal and restricted conditions such as the spherical symmetric clouds, disks, fixed boundary, equilibrium state or gases with uniform density, and then using linearized or numerical analysis. For example, the most famous \textit{Jeans instability} (see, for instance, \cite{LeBlanc2010,Tomisaka2012,Ward-Thompson2011}) gives Jeans mass and radius limits and criterion for the collapse of the extended, uniform (density) and static clouds in the earliest stage of the fragmentation. However, in practice, such ideal molecular clouds have never been observed because of the strong and unreal assumptions of the shape, linearity and density distributions.

In reality, the situation, according to the observations by astronomers (see the references above), is much more complex and very far from the uniform-density gas clouds envisaged by Jeans. Inhomogeneities give rise to more than one center
of attraction, and molecular clouds are generally irregularly shaped (usually they are comprised of dense molecular cloud cores and filaments structures, see \cite[\S $4.3.3$]{Ward-Thompson2011}), and do not at all resemble equilibrium configurations, see \cite{Larsona} and will be
rotating. We emphasize that rotations must interfere with the gravitational collapse due to the fictitious inertia centrifugal force in a rotating frame of reference conflicts with the gravity. In addition, interstellar magnetic fields may deflect particles and obstruct their attraction towards the center. Therefore, in order to understand the star formation, it is worth understanding the evolution of molecular clouds \textit{without above ideal assumptions} in the first place.

Although the evolution of giant molecular clouds is a key
ingredient for understanding of the star formation and there are numerous observational, numerical and theoretical works on this topic in \textit{astrophysical literature}, the main formation mechanics and the origin of their physical conditions still remains uncertain and is an unsolved
problem to date due to their extreme complexity (too flexible to capture the main properties) and the outstanding discrepancy between the observations and the theoretical models. In the meanwhile, to the best of our knowledge there is no literature \textit{in mathematics} to rigorously model and analyze the fully nonlinear evolution of giant molecular clouds (it has been pointed out by Rendall \cite[\S $7.2$]{Rendall2002} that there are no
results on Jeans instability available for the fully nonlinear case).  We, in this article, attempt to make the first step on the nonlinear analysis of the large, expanding and irregularly-shaped self-gravitational molecular clouds and give possibilities of the star formation (mass accretions), fragmentation of the molecular clouds and other singularities due to the diffuse boundary.

\subsection{The diffuse boundary problem of molecular clouds}
\label{s:1.1}

Let us begin with stating the following model of molecular clouds.

$(i)$ \underline{Variables:} The molecular clouds are characterized by following variables: We use functions $\rrho:[0,T)\times \Rbb^3\rightarrow \Rbb_{\geq 0}$, $ p:[0,T)\times \Rbb^3\rightarrow \Rbb_{\geq 0}$  and $\rPhi:[0,T)\times\Rbb^3 \rightarrow \Rbb$, for some constant $T>0$, to describe the distribution of the mass density, pressure of the fluids and the Newtonian potential of gas clouds, respectively, and denote $\Omega(t):=\supp\rrho(t,\cdot) =\{\bx\in \Rbb^3\;|\;\rrho(t,\bx)>0\}\subset \Rbb^3$ is the changing volume occupied by the gas at time $t$. Then the vacuum is identified by $\rrho(t,\bx)=0$ in $\Omega^\mathsf{c}(t)$ and the fluids by $\rrho(t,\bx)>0$ in $\Omega(t)$. We focus on the \textit{isentropic ideal} gas throughout this article, that is, the \textit{equation of state} bridging the pressure $\rp$ and the density $\rrho$ of the gas cloud is given by
\begin{align}\label{e:eos1}
 p= K  \rrho^{\gamma}  \quad \text{for} \quad \bx\in \Rbb^3
\end{align}
where $\gamma>1$ and $K\in \Rbb_{>0}$ are both given constants. The velocity of fluids is only defined on $\Omega(t)$ and denoted by  $\mathring{\mathbf{w}}:=(\mathring{w}^i):[0,T)\times \Omega(t)\rightarrow \Rbb^3$ (note that there is no definition in $\Omega^\mathsf{c}(t)$).

We call $\Omega(t)$ the \textit{hydrodynamic region} and $\Omega^\mathsf{c}(t)$ the \textit{ballistic region} (see \cite{Liu2021d}). In ballistic region, rarefied gas dynamics prevails and rarefied gas means there is no mass accretion, which, therefore, does not interest us due to the motivations of star formations. However, in the hydrodynamic region $\Omega(t)$, the regime of hydrodynamics applies, and in order to clarify the evolution of the hydrodynamic region of clouds, we propose the following Euler--Poisson system with the diffuse boundary to frame out the hydrodynamic region and characterize this system. Later it turns out this system is well-posed and we emphasize again that we are interested in the potion of the hydrodynamic region of molecular clouds.

$(ii)$ \underline{Euler--Poisson equations} determine the developments of molecular clouds in $\Omega(t)$ and the Newtonian potential on $\Rbb^3$, that is,
\begin{align}
\del{0} \rrho + \mathring{w}^i \del{i} \rrho+ \rrho \del{i} \mathring{w}^i = & 0  \quad  &&\text{in}\quad\Omega(t),  \label{e:NEul1} \\
\rrho \del{0} \mathring{w}^k + \rrho \mathring{w}^i \del{i} \mathring{w}^k + \delta^{ik} \del{i}  p = & -   \rrho \del{}^k  \Phi &&\text{in}\quad \Omega(t),  \label{e:NEul2} \\
\Delta  \Phi =  &  \rrho,
 &&\text{in}\quad \Rbb^3,  \label{e:NEul3}
\end{align}
for $t\in [0,T)$. The Newtonian potential $ \Phi$ is given by
\begin{equation}\label{e:Newpot}
 \Phi(t,\bx)=-\frac{1}{4\pi}\int_{\Omega(t)}				 \frac{\rrho(t,\by)}{|\bx-\by|}d^3 \by \quad \text{for} \quad  (t,\bx)\in [0,T) \times \Rbb^3,
\end{equation}

$(iii)$ \underline{Initial data} are prescribed by
\begin{align}\label{e:inidata}
\rrho(0,\bx)=\rrho_0(\bx) \quad \text{for} \quad \bx\in \Rbb^3 \AND \rw^k(0,\bx)= \rw^k_0(\bx) \quad \text{for} \quad \bx\in \Omega(0).
\end{align}
We assume $\Omega(0)\subset \Rbb^3$ is a precompact set throughout this article.

$(iv)$ \underline{Diffuse boundary} is defined by
\begin{equation}\label{e:ffbdry0}
   \lim_{(t, \bx)\rightarrow (\check{t},\check{\bx}) }\bigl( \rrho^{-1} (t,\bx)\del{i} \rp  (t,\bx)\bigr) =0,
\end{equation}
for any $(t, \bx) \in [0,T)\times \Omega(t)$ and $(\check{t},\check{\bx}) \in [0,T)\times \partial\Omega(t)$ where $\partial \Omega(t)$ is the boundary of the volume $\Omega(t)$, that is, the moving interface between the cloud and the exterior vacuum.
This boundary condition characterizes the gradually vanishing fluid pressure per mass near the boundary.

$(v)$ \underline{Classical solution:}
We define classical solutions of the diffuse boundary problem $(i)$--$(iv)$:
\begin{definition}\label{t:clsl}
	A set of functions $(\rrho,\rw^i,\rPhi,\Omega(t))$ is called \textit{a classical solution} to above problem $(i)$--$(iv)$ on $t\in[0,T)$ for $T>0$ if it solves the system $(i)$--$(iv)$ and satisfies the following conditions:
	\begin{enumerate}
		\item \label{D:1} there is a \textit{$C^1$-extension} $w^i$ of $\rw^i$ to the boundary, i.e., there exists 
$w^i\in C^1([0,T)\times \overline{\Omega(t)},\Rbb^3)$, such that $w^i(t,\bx)=\rw^i(t,\bx)$ for $(t,\bx)\in[0,T)\times \Omega(t)$;
		\item \label{D:2}  $\rrho\in C^1([0,T)\times \Rbb^3,\Rbb)$ and $\rPhi$ is given by \eqref{e:Newpot}.
	\end{enumerate}
\end{definition}

\begin{remark}\label{t:diff} 
	We emphasis the \textit{main difficulty} of the blowups of the irregular-shaped fluids  with vacuum (molecular clouds) comes from the initial irregular configurations of its density (multiple centers of gravity). This inevitably causes far more complicated movements during the evolution than the one with spherical symmetry. The spherical symmetric case has been well studied. However, for the non-spherical symmetric case with the gravity, it is very rare in the references due to a very complicate nonlinear interaction between the gravity and the fluids.
\end{remark}

\subsection{Notation and convention}

\subsubsection{Vectors and components} \label{s:vecnota}
We will use, unless otherwise stated, boldface, e.g., $\bx$, $\by$, $\mathbf{w}$ for Latins and the Greeks, e.g., $\xi$, to denote vectors and normal font with indices, e.g., $x^i$, $y^i$, $\xi^i$ and $w^i$, to denote the components of vectors (we also use these components to express the vector if it clear from the context).

\subsubsection{Indices and summation convention}\label{iandc}
Throughout this article, unless stated otherwise, we adopt the notation system in general relativity (i.e. \textit{Einstein notation}, see \cite{Wald2010} for more). In specific, we use upper indices to represent components of vectors,
lower indices to represent components of covectors, and use lower case Latin letters, e.g. $i, j,k$, for spatial indices that run from $1$ to $3$, and lower case Greek letters, e.g. $\alpha, \beta, \gamma$, for spacetime indices
that run from $0$ to $3$ (we also denote time coordinate $t$ by $x^0:=t$). We use \textit{Einstein summation convention} as well. That is, when an index variable (i.e. dummy index) appears twice, once in an upper superscript and once in a lower subscript, in a single term and is not otherwise defined, it implies summation of that term over all the values of the index. For example, for a summation $z=\sum^3_{i=1} x^i y_i$, we simply use Einstein notation to denote $z= x^i y_i$. In addition, we raise and lower indices by the Euclidean metric $\delta^{ij}$ and $\delta_{ij}$, that is, for example,
\begin{equation*}
	W^{ij}:=W_{lk}\delta^{li}\delta^{kj}, \quad W_{ij}:=W^{lk}\delta_{li}\delta_{kj} \AND W^{ij}:=W^i_k\delta^{kj}.
\end{equation*}

\subsubsection{Lagrangian descriptions}\label{s:lag}
Suppose a field $f:[0,T)\times \Omega(t) \rightarrow V$ (or a property $f$ in Eulerian description) and the flow $\chi:[0,T)\times \Omega(0)\rightarrow \Omega(t)\subset \Rbb^3$ generated by a vector field  $\mathbf{w}:=(w^i)$, such that  $\chi(t,\xi)=\bx \in \Omega(t)$ for every  $(t,\xi)\in[0,T)\times \Omega(0)$ where $T>0$ is a constant, $\Omega(t) \subset \Rbb^3$ is a domain depending on $t$ and $V \subset \Rbb^n$ for some $n \in \Zbb_{\geq 1}$, we denote
\begin{equation}\label{e:undf}
	\underline{f}(t,\xi):=f(t,\chi(t,\xi))=f(t,\bx)
\end{equation}
describing the property $\underline{f}$ of the parcel labeled by the initial position $\xi$ at time $t$ (the property $\underline{f}$ along the flow, i.e., in Lagrangian description). Using this notation, we have
\begin{equation*}
	\del{t}\underline{f}(t,\xi)=\underline{D_t f  }(t,\xi)
\end{equation*}
where $D_t$ is the \textit{material derivative} (i.e. $D_t:=\del{t}+w^i\del{i}$, see, for instance, \cite{Chorin1993}).
According to the definition of $L^\infty$, if $\chi(t,\Omega(0))=\Omega(t)$, we conclude that
\begin{equation*}
	\|f(t)\|_{\Li(\Omega(t))}=\|\underline{f}(t)\|_{\Li(\Omega(0))}.
\end{equation*}

\subsubsection{Matrices}\label{s:mtrx}
For any matrices $A$ and $B\in \mathbb{M}_{n\times n}$, we define
\begin{equation*}
	A\leq B \quad \Longleftrightarrow \quad \xi^T A \xi \leq \xi^T B \xi
\end{equation*}
for any column vector $\xi\in \Rbb^n$ and $\xi^T$ is its transpose. We can similarly define $A<(>,\geq,\lesssim,\gtrsim)B$.

\subsection{Preliminary  concepts}\label{s:defs}
In order to state the main theorem concisely, we first introduce a few definitions.
\subsubsection{Decomposition of velocities}Let us first define
$z$ and $X$ related to the radial component of velocity and its perpendicular component, respectively,
\begin{align}
	z(t,\xi):=& \delta_{ij}  \frac{\chi^j(t,\xi)}{|\chi(t,\xi)|} \underline{w}^i(t, \xi)  \label{e:zdef}
	\intertext{and}
	X(t,\xi) := &  \sqrt{\delta_{jk} \underline{w}^j(t, \xi ) \underline{w}^k(t, \xi ) - z^2(t, \xi )}. \label{e:Xdef}
\end{align}
For simplicity of notations, we denote
\begin{equation*}\label{e:defW}
	W^i_j(t,\bx):=w^i_{,j}(t,\bx):=\begin{cases}
		\del{j}\rw^i(t,\bx) \quad & \text{if}\quad (t,\bx)\in[0,T)\times \Omega(t)\\
		\lim_{(t^\prime,\bx^\prime)\rightarrow(t,\bx)}\del{j}\rw^i(t^\prime,\bx^\prime)\quad & \text{if}\quad (t,\bx)\in[0,T)\times \del{}\Omega(t)
	\end{cases}.
\end{equation*}
Let us lower the index of $W^i_j$ by $\delta_{ki}$, i.e., $W_{jk}(t,\bx):=  \delta_{ki}W^i_j(t,\bx)$ and decompose
\begin{equation}\label{e:WThOm0}
	W_{jk}(t,\bx)
	=   \frac{1}{2}\bigl(W_{jk}(t,\bx)+W_{kj}(t,\bx)\bigr)+\frac{1}{2}\bigl(W_{jk}(t,\bx)-W_{kj}(t,\bx)\bigr) 
	=   \Theta_{jk}(t,\bx)-\Omega_{jk}(t,\bx)
\end{equation}
where $\Theta_{jk}$ is the symmetric deformation component ($\Theta_{jk}=\Theta_{kj}$) of $W_{jk}$ and $\Omega_{jk}$ the antisymmetric rotation component ($\Omega_{jk}=-\Omega_{kj}$) defined by
\begin{align}
	\Theta_{jk}(t,\bx):=& \frac{1}{2}\bigl(W_{jk}(t,\bx)+W_{kj}(t,\bx)\bigr) \label{e:WThOm0a}\\
	\Omega_{jk}(t,\bx):=& \frac{1}{2}\bigl(W_{kj}(t,\bx)-W_{jk}(t,\bx)\bigr). \label{e:WThOm0b}
\end{align}
Then we have the identities,
\begin{align}\label{e:WThOm1}
	W_{kj}=\Theta_{jk}+\Omega_{jk} \AND W_{jk}=\Theta_{jk}-\Omega_{jk}.
\end{align}
We also denote the divergence of the velocity
\begin{equation}\label{e:theta}
	\Theta(t,\bx):=\delta^{jk}\Theta_{jk}(t,\bx)=\delta^{jk}W_{jk}(t,\bx)= w^j_{,j}(t,\bx).
\end{equation}

\subsubsection{Mass, energy, moment of inertia and virial}

Let us first define the total mass, energy, moment of inertia and virial of the molecular clouds.
Firstly, $ M $ is the \textit{total mass} given by
\begin{equation}\label{e:mass}
	M (t)= \int_{\Rbb^3} \rrho d^3\bx,
\end{equation}
the \textit{total energy} is defined by
\begin{align}\label{e:eng0}
    E (t) = & \int_{\Rbb^3} \biggl( \frac{1}{2} \rrho |\mathring{w}^i|^2+\frac{p}{\gamma-1}  +\frac{1}{2}  \rrho \Phi \biggr)d^3\bx
	=   \int_{\Rbb^3} \biggl( \frac{1}{2}  \rrho |\mathring{w}^i|^2+\frac{ K\rrho^\gamma }{\gamma-1} +\frac{1}{2}  \rrho \Phi \biggr)d^3\bx,
\end{align}
the \textit{moment of inertia} and  \textit{virial}\footnote{See \eqref{e:dtH} for the reason of the notation $H^\prime(t)$. } are given, respectively, by
\begin{align}\label{e:xi0}
	H (t):=\frac{1}{2}\int_{\Rbb^3}\rrho(t,\bx) |\bx|^2 d^3 \bx \AND
	H^\prime (t):=\int_{\Omega(t)}    \rrho  \mathring{w}^i x^j  \delta_{ij} d^3\bx.
\end{align}
\begin{remark}
	By the conservations of the energy and mass, see Lemma \ref{t:ctoms} in \S\ref{s:consv}, the energy $E(t)$ and mass $M(t)$ are conserved. Thus, we denote $E:=E(t)=E(0)$ and $M:=M(t)=M(0)$ throughout this article. 	
\end{remark}

\subsubsection{The regular distributions and first class of global solutions}
The regular behavior near the  boundary (i.e., $\mathrm{R}^{b}([0,T)\times \Omega(t),G_b)$) describes a type of the behavior of density near the boundary $\{t\}\times \del{}\Omega(t)$. Using this concept, by excluding a class of global solutions (i.e., the first class of global solutions below), intuitively, we can make sure the gravity of clouds near the boundary satisfies the \textit{law of inverse square}, $\sim 1/|\bx|^2$ (i.e., if $\rrho\in \mathrm{R}^{1}([0,T)\times\Omega(t),G_1)$ given below) in the certain sense, which is crucial in the proof of the main theorem in \S\ref{s:bdyevo}, and the tidal force near the boundary satisfies $\sim 1/|\bx|^3$ (i.e., if $\rrho\in \mathrm{R}^{0}([0,T)\times\Omega(t),G_0)$).
Next, we first introduce the regular distribution of the density and then the first class of global solutions.
\begin{definition}\label{t:diffdis}
	Suppose $b=0$ or $1$ is a constant, $G_b>0$ and $T\in(0,\infty]$ are given constants and $\Omega(t)$ are given sets for every $t\in[ 0, T)$,
	a function $\rrho: [0, T) \times \Omega(t)\rightarrow \Rbb$ is called \textit{regular near the boundary with parameter $(b, G_b)$} and denoted by $\rrho\in \mathrm{R}^b([0,T)\times \Omega(t),G_b)$, if there are constants  $\delta\in(0,1)$, such that 	
	\begin{equation*}
		G_{b}> \Bigl(\frac{1}{2(b+1)  \delta^{3-b}} +   
		\frac{3(2-b) }{4}\Bigr)	\frac{M}{\pi},
	\end{equation*}
and
	for any  $t\in[0, T)$, every vector $\bx \in\del{}\Omega(t)$ and $\mathbf{r}\in B(\mathbf{n},\delta)$ where  $\mathbf{n}:=\bx/|\bx|$,
\begin{align}\label{e:regdef0}
	\rrho(t, |\bx|\mathbf{r})<\frac{3M |\mathbf{n}-\mathbf{r}|^{1-b}}{4\pi \delta |\bx|^{3}}.
\end{align}
\end{definition}

\begin{definition}\label{t:12class}
	Let $(\rrho,\rw^i,\Omega(t))$ is a global solution to the diffuse boundary problem of Euler--Poisson equations \eqref{e:NEul1}--\eqref{e:ffbdry0} for $t\in [0,\infty) $. If $\rrho \notin   \mathrm{R}^1([0,\infty)\times\Omega(t),G_1)$, then we call this global solution \textit{the first class of global solution with parameter $G_1$}. More specifically, for any constants $\delta\in(0,1)$, such that 	
	\begin{equation*}
		G_{1}> \Bigl( \frac{1}{ \delta^{2}} + 3\Bigr) \frac{M}{4\pi},
	\end{equation*}
	there is at least a time $T_\star\in[0,\infty)$, vectors $\bx \in\del{}\Omega(t)$ and $\mathbf{r}_\star\in B(\mathbf{n},\delta)$ where  $\mathbf{n}:=\bx/|\bx|$, such that
	\begin{align}\label{e:rhobg}
		\rrho(T_\star, |\bx|\mathbf{r}_\star) \geq \frac{3M }{4\pi\delta|\bx|^{3}} .
	\end{align}

If we denote
\begin{equation*}
	\bar{\rho}(t):=\frac{M}{\frac{4}{3}\pi(\sup_{\bx\in\del{}\Omega(t)}|\bx|)^3} \AND \varrho(t,\bx):=\frac{\rho(t,\bx)}{\bar{\rho}(t)}
\end{equation*}
the averaging density of  clouds and relative density, respectively. Then \eqref{e:rhobg} implies
\begin{align}\label{e:rhobg2}
	\varrho(T_\star, |\bx|\mathbf{r}_\star) \geq \frac{1}{\delta}>1,
\end{align}
we call this phenomenon as the \textit{near-boundary mass accretion}.  
\end{definition}
\begin{remark}
	We call the phenomenon \eqref{e:rhobg2} the near-boundary mass accretion since it characterizes the \textit{locally focusing effects of density near the boundary}. In astrophysics, if $G_1$ is large enough, then $\delta$ can become extremely small, then the relative density locally (in a local ball) near the boundary becomes extremely large by \eqref{e:rhobg2}. This means the fluids are accumulating locally near the boundary and this exactly describes the \textit{formations of protostars}. 
\end{remark}

\begin{remark}
	Later, we will see in \S\ref{s:estNewpot} and \S\ref{s:tdal} that if $\Phi$ is the Newtonian potential given by \eqref{e:Newpot}, and for any $(t, \bx) \in [0,T) \times  \del{}\Omega(t) $, then the Newtonian gravity (if $\rrho \in  \mathrm{R}^{1}([0,T)\times\Omega(t),G_1)$) and tidal force (if $\rrho \in  \mathrm{R}^{0}([0,T)\times\Omega(t),G_0)$) satisfy, respectively,
	\begin{align}\label{e:grtdl}
		|\del{i} \Phi(t,\bx) |	\leq     \frac{G_1}{|\bx|^{2}} \AND |\del{i}\del{j} \Phi(t,\bx) |	\leq     \frac{G_0}{|\bx|^{3 }}.
	\end{align}
	Therefore, we call constants $G_1$ and $G_0$ the \textit{gravity bound} and \textit{tidal force bound},  respectively.
\end{remark} 
\begin{remark}\label{t:rmk1}
	By Definitions \ref{t:diffdis}, if $G_0=2G_1$, then  $\mathrm{R}^0([0,T)\times\Omega(t),G_0) \subset \mathrm{R}^1([0,T)\times\Omega(t),G_1)$.
\end{remark}

\subsubsection{Initial data}
We denote the initial data by $\mathbf{w}_0(\xi):=\mathbf{w}(0,\xi)$,
\begin{gather*}
	z_0(\xi):= z(0,\xi)  \AND
	X_0(\xi):= X(0,\xi) , \\
	\Theta_0(\xi):=\Theta(0,\xi), \quad \Omega_{0jk}(\xi):=\Omega_{jk}(0,\xi) \AND \Xi_{0jk}(\xi):=\Xi_{jk}(0,\xi),
\end{gather*}
and denote
\begin{equation*}
	|\Xi_{0jk}(\xi) |:=\sqrt{\Xi_{0jk}(\xi)\delta^{ij}\delta^{kl}\Xi_{0il}(\xi)} \AND |\Omega_{0jk}(\xi) |:=\sqrt{\Omega_{0jk}(\xi)\delta^{ij}\delta^{kl}\Omega_{0il}(\xi)}.
\end{equation*}

\begin{definition}\label{t:paramtr}
	\begin{enumerate}[(1)]
		\item \label{t:Def1} For any given energy $E>0$,  gravity bound $G_1>0$,  tidal force bound $G_0>0$ and initial virial $H^\prime(0)$, define small constants
		\begin{align}\label{e:sigma0}
			\sigma_\star:= \min\Bigl\{\frac{1}{5},\frac{\beta E}{500|H^\prime(0)|}\Bigr\} \AND \sigma_\dag:=	\min\left\{  \frac{1}{263200^4}, \sigma_\star\right\}
		\end{align}
		where $\beta=\min\{1,3(\gamma-1)\}>0$ (recall $\gamma$ is given by \eqref{e:eos1}).
		If for a constant $\sigma>0$ (and $\sigma\in(0,\sigma_\star)$ for later \eqref{t:Def3}, admissible initial boundary data;  $\sigma\in(0,\sigma_\dag)$ for later \eqref{t:Def4}, strong admissible initial boundary data),  $\lambda_0:=\lambda_0(\sigma)$, $\lambda_1:=\lambda_1(\sigma)$ are in the following intervals
		\begin{align*}
	\lambda_0 \in   \Bigl(\frac{ 2  }{2- \sigma} , \frac{1}{1-\sigma} \Bigl(1+ \frac{1}{14}\sigma\Bigr)   \Bigr) \AND
			\lambda_1 \in   \Bigl((1-\sigma) \lambda_0^2 , \Bigl(1+ \frac{1}{14}\sigma\Bigr)\lambda_0  \Bigr),
		\end{align*}
	then the set of parameters $( \lambda_0,\lambda_1)$ is called \textit{admissible parameters} .
	\item \label{t:Def2} For any given $E>0$, initial virial $H^\prime(0)$ and gravity bound  $G_1>0$, we call $r_c$ the \textit{critical radius} of molecular clouds if it is defined by
		\begin{align}\label{e:crds}
			r_c: =& \frac{2 (9G_1)^\frac{1}{3}}{(2-\sigma_\star)\sigma_\star}
		\end{align}
		\item \label{t:Def2b} We call the set of energy, mass and gravity bound $(E,M,G_1)$ is \textit{compatible}, if the total energy $E>0$, total mass $M>0$ and gravity bound $G_1>0$ (see Definition \ref{t:diffdis}) satisfy a relation	
		\begin{equation*}\label{e:GEM}
			(9G_1)^{\frac{1}{3}}<\frac{1}{24}\Bigl(\frac{\beta E}{M}\Bigr)^{\frac{1}{2}} .
		\end{equation*}
		\item \label{t:Def3} Suppose  $(E,M,G_1)$ is a given compatible set  and $H^\prime(0)$ is a given constant, we call the initial boundary velocity $\mathbf{w}_0|_{\del{}\Omega(0)}$ and the initial domain $\Omega(0)$, denoted by $(\mathbf{w}_0|_{ \del{}\Omega(0)}, \Omega(0))$, is the \textit{admissible initial boundary data}, if $ \Omega(0)\supsetneq \overline{B(0,r_c)}$,
		and for any $\xi\in\del{}\Omega(0)$ and any given radial velocity field $z_0(\xi)$, there is a small constant $\sigma_\xi\in(0,\sigma_\star)$ where $\sigma_\star$ is given by \eqref{e:sigma0}, such that they satisfy
		\begin{align}
			&\hspace{-4cm} (A)\quad  \Omega(0) \text{ is a precompact set and }  |\xi|>  \lambda_0 (9G_1 )^{\frac{1}{3}} \sigma_\xi^{-1} >r_c; \label{e:BinO} \\
			&\hspace{-4cm} (B)\quad z_0(\xi):=\lambda_1\lambda_0^{-1}\sigma_\xi |\xi| \in \biggl(\lambda_1(9G_1)^{\frac{1}{3}},  \frac{\lambda_1}{24} \sqrt{ \frac{ \beta E }{ M} } \biggr);  \label{e:xidn0} \\
			&\hspace{-4cm} (C)\quad   X_0(\xi)\in \biggl[0, \frac{\sigma_\xi}{\lambda_1}\sqrt{\frac{\lambda_0}{2}} z_0(\xi)\biggr) \label{e:Xdn}
		\end{align}
		where $( \lambda_0,\lambda_1)=( \lambda_0(\sigma_\xi),\lambda_1(\sigma_\xi))$ is admissible parameters,
		and denote the set of all the admissible initial boundary data  by
		\begin{align*}
			&
			\mathfrak{A}(E,M,G_1,H^\prime(0)):= \notag  \\
			&\hspace{0.5cm} \bigl\{(\mathbf{w}_0|_{ \del{}\Omega(0)}, \Omega(0))\;|\; (\mathbf{w}_0|_{ \del{}\Omega(0)}, \Omega(0)) \text{ is an  admissible initial boundary data} \bigr\}.
		\end{align*}	
		\item \label{t:Def4} Suppose  $(E,M,G_1)$ is a given compatible set, $H^\prime(0)$, $G_0>0$  and $\sigma\in(0,\sigma_\dag)$ are given constants where $\sigma_\dag$ is given by \eqref{e:sigma0}, we call the initial boundary velocity $\mathbf{w}_0|_{\del{}\Omega(0)}$ and the initial domain $\Omega(0)$, $(\mathbf{w}_0|_{ \del{}\Omega(0)}, \Omega(0))$, is the \textit{strong admissible initial boundary data}, if
		\begin{align}
			&
			\overline{B\biggl(0,\lambda_0 (9G_1 )^{\frac{1}{3}} \sigma^{-1}\biggr)}\subsetneq \Omega(0)\subsetneq B\biggl(0,\frac{\lambda_0}{24} \sqrt{\frac{\beta E}{M}}\sigma^{-1}\biggr) ;
			 \label{e:BinO.a}
			 \end{align}
		 and for any $\xi\in\del{}\Omega(0)$, the initial velocity satisfies
		 \begin{align}
		 	& (B^\star)\quad z_0(\xi) =  \lambda_1\lambda_0^{-1}\sigma  |\xi|;  \label{e:xidn0.a} \\
			& (C^\star)\quad   X_0(\xi)\in \biggl[0, \frac{\sigma}{\lambda_1}\sqrt{\frac{\lambda_0}{2}} z_0(\xi)\biggr); \label{e:Xdn.a}  \\
			& (D^\star)\quad  \sqrt{\frac{3\lambda_1}{4\lambda_0}\Bigl|\Theta_0^{-1}(\xi)-\frac{\lambda_0}{3 \lambda_1} \sigma^{-1} \Bigr|}+\Theta_0^{-\frac{7}{4}}(\xi)|\Xi_{0jk}(\xi) |+\Theta_0^{-2}(\xi)|\Omega_{0jk}(\xi) | < \frac{1}{4}\sigma^{-\frac{1}{2}} \label{e:deldata.a}
		\end{align}
		where $( \lambda_0,\lambda_1)=( \lambda_0(\sigma),\lambda_1(\sigma))$ is admissible parameters, and denote the set of all the strong admissible initial boundary data by
		\begin{align*}
			&\mathfrak{A}^\dag(E,M,G_1,G_0,H^\prime(0),\sigma):=  \notag  \\
			&\hspace{0.5cm} \bigl\{(\mathbf{w}_0|_{ \del{}\Omega(0)}, \Omega(0))\;|\; (\mathbf{w}_0|_{ \del{}\Omega(0)}, \Omega(0)) \text{ is an strong admissible initial boundary data} \bigr\}.
		\end{align*}	
	\end{enumerate}
\end{definition}
\begin{remark}
Condition $ \Omega(0)\supsetneq \overline{B(0,r_c)}$ of Definition \ref{t:paramtr}.\eqref{t:Def3} means the cloud is \textit{large} enough and $(B)$ implies the cloud is \textit{expanding} initially. We point out the (strong) admissible initial boundary data are far from optimal.
\end{remark}
\begin{remark}\label{r:wlldefad}
	We point out if $ \Omega(0)\supsetneq \overline{B(0,r_c)}$ and $(E,M,G_1)$ is a compatible set, then the interval in \eqref{e:xidn0} is nonempty for every $\xi\in\del{}\Omega(0)$, and further there exists admissible initial boundary data.
\end{remark}

\begin{remark}\label{r:wlldefad2}
	Any strong admissible initial boundary data is admissible initial boundary data. See Lemma \ref{t:sadad} for details.
\end{remark}

\subsection{Main Theorem}
After above concepts, we are in a position to present the main theorem. This theorem states the local existence, uniqueness, continuation principle and the finite time blowups of classical solutions to the diffuse problem in certain conditions.
\begin{theorem}[Main Theorem]
	\label{t:mainthm}
	Suppose the initial data $(\rrho_0, \rw^i_0)$ of the diffuse boundary problem is given by  \eqref{e:inidata}, $\Omega(0)$ is a precompact $C^1$-domain, $\rrho_0\in C^1( \Rbb^3,\Rbb_{\geq 0}) $, $\rw^i_0 \in C^1(\Omega(0),\Rbb^3) $. If
	\begin{equation*}
		(a) \quad 1<\gamma\leq \frac{5}{3}, \quad (\rrho_0)^{\frac{\gamma-1}{2 }} \in H^3(\Rbb^3) \AND \rw^i_0\in H^3(\Omega(0),\Rbb^3),
	\end{equation*}
	or if
	\begin{align*}
		(b)\quad 1<\gamma<2, \quad  \rrho_0 \in H^3(\Rbb^3), \quad (\rrho_0)^{\frac{\gamma-1}{2 }}
		\in H^4(\Rbb^3) \AND
		\rw^i_0\in H^4(\Omega(0),\Rbb^3),
	\end{align*}
	then:

$(1)$ \emph{(Weak blowups)} further suppose   $(E,M,G_1)$ is any given compatible set and  $H^\prime(0)$ is any given constant, and there is no the first class of the global solution with  $G_1>0$, then
if the initial boundary data $(\mathbf{w}_0|_{ \del{}\Omega(0)}, \Omega(0)) \in \mathfrak{A}(E,M,G_1,H^\prime(0))$ is admissible initial boundary data, then there is no the global solution to the diffuse boundary problem, and there is a constant $0<T^\star<+\infty$, such that, the classical solution breaks down at $t=T^\star$ and
\begin{align*} 
	\int^{T^\star}_0 \Bigl(\|\nabla w^i(s)\|_{L^\infty( \Omega(s))}  +\|\nabla \rrho^{\frac{\gamma-1}{2}}\|_{ L^\infty(\Omega(s))} \Bigr) ds=+\infty;
\end{align*}

$(2)$ \emph{(Global solution and mass accretions near the boundary)} further suppose   $(E,M,G_1)$ is any given compatible set and  $H^\prime(0)$ is any given constant and the initial boundary data $(\mathbf{w}_0|_{ \del{}\Omega(0)}, \Omega(0)) \in \mathfrak{A}(E,M,G_1,H^\prime(0))$ is admissible initial boundary data,  if $(\rrho,\rw^i,\Omega(t))$ is a global solution to the diffuse boundary problem of Euler--Poisson equations \eqref{e:NEul1}--\eqref{e:ffbdry0} for $t\in [0, +\infty) $, then it has to be the first class of global solution with $G_1$ and includes near-boundary mass accretions.

$(3)$ \emph{(Strong blowups)}
further suppose   $(E,M,G_1)$ is any given compatible set,   $H^\prime(0)$, $G_0=2G_1>0$ and $\sigma\in(0,\sigma_\dag)$, where $\sigma_\dag$ is given by \eqref{e:sigma0}, are given constants, the solution is not the first class of global solution with $G_1$ and if initial data are given by $(b)$,
then
if $(\mathbf{w}_0|_{ \del{}\Omega(0)}, \Omega(0)) \in \mathfrak{A}^\dag(E,M,G_1,G_0,H^\prime(0),\sigma)$ is the strong admissible initial boundary data, then there is a constant $0<T^\star<+\infty$, such that,
the classical solution to the diffuse boundary problem breaks down at $t=T^\star$,
and if $T^\star\in(0,T^\natural]$, where $T^\natural$ is the  supercritical time given by Corollary \ref{e:supcri}, and $\rrho\in \mathrm{R}^0([0,T^\star)\times \Omega(t),G_0)$, then there is a small constant $\epsilon>0$, such that
\begin{equation*}\label{e:blup2}
	\int^{T^\star}_0 \Bigl( \|\Theta(s)\|_{L^\infty( \mathring{\Omega}_\epsilon(s) )}  +\|\Omega_{jk}(s)\|_{ L^\infty( \mathring{\Omega}_\epsilon(s) )}+\|\nabla \rrho^{\frac{\gamma-1}{2}}(s)\|_{ L^\infty(\Omega(s))} \Bigr) ds=+\infty.
\end{equation*}
\end{theorem}

\begin{remark}
	We emphasize that comparing with $(3)$, $(4)$ removes some singularities on the boundary.
\end{remark}

\subsection{Related work}
As we mentioned before,  the motivation of this work is to figure out some possible mechanisms of star formations, while our previous paper \cite{Liu2018,Liu2018b,Liu2018a} indicate the stability of the universe. However,  \cite{Liu2018a} gives the possibilities of the instability. In our recent work \cite{Liu2022}, we proved the slightly nonlinear Jeans instability, which may lead to the formations of stellar system. 
 
In the absence of vacuum, the density is bounded below from zero everywhere, then the theory of symmetric hyperbolic systems \cite{Friedrichs1954,Kato1975,Lax1954,Klainerman1981,Klainerman1982,Majda2012} is available. Formation of singularities are also widely researched. Sideris \cite{Sideris1985} introduced an averaging quantity (in fact, it is the virial theorem), and by connecting this quantity with the Euler equations, lifespans can be derived, but this method can not reveal the detailed behavior of blowups. This quantity will be adopted in this article as well.  See Alinhac \cite{Alinhac1995} and the references cited therein for the details about more methods, including the spirits of his groundbreaking work, on blowup for nonlinear hyperbolic and wave equations. 

The Euler--Poisson system with vacuum was started with Makino \cite{Makino1986,Makino1987}, which gave the local existence theorem by introducing the Makino's density and the unphysical concept of tame solutions. Later on, a series of works  \cite{Makino1990,Makino1986a,Perthame1990,Makino1992,Liu1997} continued to study the evolution of tame solutions to Euler equation with or without gravity. Note that the blowup results of Euler--Poisson system with vacuum \cite{Makino1990,Perthame1990,Makino1992} of tame solutions are under the condition of spherical symmetry, but there is no symmetric assumptions for Euler equations in \cite{Makino1986a}. The main tool of proofs of blowup in these papers involves Sideris' averaging quantity \cite{Sideris1985} (the virial theorem). The advantage of this method is easy to prove blowup, but usually, it does not give more information about the detailed behavior of the finite time blowup solutions. We point out that \cite{Makino1992} specified \textit{Makino's conjecture} that any tame
solution, including nonsymmetric tame solution, will become not tame after a finite time, and this current article can be viewed as a partial answer. Oliynyk \cite{Oliynyk2008a},  Brauer and Karp \cite{Brauer2018} and the references cited therein have improved the local existence and uniqueness of solutions to  Euler--Poisson(--Makino) system with vacuum. Brauer \cite{Brauer1998} gives a strong continuation principle for tame solutions, especially, a more detailed classification of the type of blowups.

Another famous model with vacuum is the one with the \textit{physical vacuum boundary condition}, which stated with \cite{Liu1996,Liu1997,Liu2000},a series works proceeding.  For example,   \cite{Jang2009,Jang2012,Jang2015,Coutand2010,Coutand2011,Coutand2012,Luo2014} established the well-posedness for Euler equation with or without gravity for gaseous stars, respectively. Several works such as \cite{Jang2008,Jang2014,Hadzic2016} considered some aspects of the long time behavior with physical vacuum along this path. Recently, Guo,  Had{\v{z}}i{\'c} and Jang \cite{Guo2018} construct an infinite-dimensional family of collapsing solutions to the Euler-Poisson system with physical vacuum condition and  density of this solution is in general space inhomogeneous and undergoes gravitational blowup. Along with the gaseous stars, a series works on liquid body are carried on by Christodoulou, Lindblad and Oliynyk, etc., for example,  \cite{Christodoulou2000, Lindblad2003a,Lindblad2005a,Lindblad2005,Oliynyk2017}.

A short remark on the generalization of Makino's Cauchy problem \cite{Makino1986} to the general relativistic cases. First, Rendall \cite{Rendall1992} generalized Makino's method to general relativity setting with certain ideal conditions. Then Oliynyk \cite{Oliynyk2008a,Oliynyk2008,Oliynyk2009} considered this problem in more general case when he researched the Newtonian limits. Brauer and Karp \cite{Brauer2014} extended Rendall and Makino's work to a very general framework.

\subsection{Overview and main ideas}

\S\ref{s:MaCoj} dedicates to the dynamical bounds on the spreading rate of the support of the density of the molecular clouds and the weak blowup theorem \ref{t:blupthm0}. In fact, the main mechanics of blowups in this article is in the level of the \textit{virial theorem}. Due to the irregular shape of this cloud and its irregular distribution of the Newtonian potential, the rotations of cloud are usually inevitable (in contrast to the spherical symmetric cases). The motion of the cloud becomes quite complicate and it is very difficult to analyze the detailed movements. The intuitive idea to overcome this difficulty is based on a simple observation that the far-field (very large distances from the center of mass) regions of the gravitational field of the molecular cloud behaves like  the surrounding of the ``normal'' point mass if $\rrho\in \mathrm{R}^{1}([0,T)\times\Omega(t))$, that is $\rrho\in \mathrm{R}^{1}([0,T)\times\Omega(t))$ makes sure the gravity on the boundary satisfies the inverse-square law and this will be proved in Proposition \ref{t:dphix2} (\S\ref{s:estNewpot}).
Intuitively, only if $|\bx|$ ($|\bx|$ is the distance between the boundary point and the mass center) is larger enough, the behaviors of the test particles on the boundary are more like those of test particles in a gravity field of a mass point. By  \textit{virial theorem} or the Sideris' average value  $ H (t)$ given in \eqref{e:xi0} (see, for example,   \cite{Sideris1985,Makino1990}), we obtain Theorem \ref{t:bluplem} (\S\ref{s:criterion}), a criterion of blowups of classical solutions. This Lemma indicates that if the boundary of the cloud is expanding below a certain rate, then there are blowups during the development of the solution and the singularities are determined by the continuation principles.

In order to derive the boundary is expanding at this rate, we can use the previous inverse-square gravity resulted by excluding the first class of global solution to help us  conclude the estimates of trajectories of free falling particles on the boundary (Proposition \ref{t:spprbd} of  \S\ref{s:traj} and see the proof in \S\ref{s:stp1}, we briefly introduce it in the next paragraph). These estimates of trajectories yield the expanding rate of the boundary is indeed smaller than the blowup rate  and further using this criterion of blowups of classical solutions leads to the weak blowup theorem \ref{t:blupthm0} in \S\ref{s:wbt}.
In addition, we point out the weak blowup theorem \ref{t:blupthm0} in some sense answers the conjecture proposed by Makino \cite{Makino1992}  that is any tame
solution will become not tame after a finite time as well.

The key step of above weak blowup theorem \ref{t:blupthm0} is to estimate the dynamical bounds on the spreading rate of the support of the density or the trajectory of free falling particles on the boundary (since the particles on the boundary are only free-falling due to the only gravity and pressureless), that is, Proposition \ref{t:spprbd}. We, along the characteristics on the boundary, select a special set of initial data (i.e., admissible data), and the suitable variables $q$, $z$ and $Y$ defined by \eqref{e:nvar0a}, \eqref{e:nvar0b} and \eqref{e:Y} to rewrite the free-falling equation \eqref{e:otvel2} to an ordinary differential equations \eqref{e:dqeq}--\eqref{e:dyeq}. The gravity $\del{i}\Phi$ is in the order of $1/|\bx|^2$ since $\rrho\in\mathrm{R}^{1}([0,T)\times\Omega(t),G_1)$. In order to obtain the upper bound of $|\bx|$, we have to control all the upper and lower bounds of $q$, $z$ and $Y$ due to the complex structure of these ODEs \eqref{e:dqeq}--\eqref{e:dyeq}.  However, direct analyzing $q$, $z$ and $Y$ is difficult. The way to overcome this difficulty is to rewrite \eqref{e:dqeq}--\eqref{e:dyeq} in terms of new variables $\dq$, $\hU$, $\hV$ and $\hy$ given by \eqref{e:dnvar1}--\eqref{e:dnvar3} and obtain the equations of these variables. It turns out these variables $\dq$, $\hU$, $\hV$ and $\hy$ present good behaviors if choosing admissible data. These data and suitable bootstrap assumptions help us conclude the lower bound of $\dq$ and the upper bound of  $\hU$, $\hV$ and $\hy$ by bootstrap arguments, furthermore, obtain all the upper and lower bounds of $|\bx|$, $z$ and $X$.

After knowing the blowup phenomenon, we usually ask the positions and the types of the singularities. \S\ref{s:evdifpro}, by removing singularities of the boundary based on some ``nice'' boundary, makes some efforts on this purpose. In order to do so, we first estimate the tidal force to be of the order $1/|\bx|^3$ if $\rrho \in  \mathrm{R}^{0}([0,T)\times\Omega(t),G_0)$ (see \S\ref{s:tdal}). It turns out this tidal force can not significantly affect the behaviors of the free falling particles on the boundary and the effects of this tidal force is close to the case without tidal force. Then by differentiating the ``limiting momentum conservation equation'' \eqref{e:NEul2} on the boundary with respect to $x^j$ (see \eqref{e:dwev} and \eqref{e:dwev2}), we derive  Newtonian version of Raychauduri's equations in terms of the expansion $\Theta$, shear $\Xi_{jk}$ and rotation $\Omega_{jk}$ (see \eqref{e:ray1}--\eqref{e:ray3}). In order to remove the singularities satisfying $\int^{T^\star}_0 \|W_{jk}(s)\|_{L^\infty( \Omega_\epsilon(s))} ds=+\infty$, we  prove $|\del{i}w^j|<\infty$ near the boundary. The idea is first to find an proper approximation solution of the Newtonian version of Raychauduri's equations and we select an approximation solution solving a Raychauduri's system without tidal force (see \eqref{e:iWrs1a}). Then try to find  perturbed solutions close to this approximation one and the perturbed solutions solve the Newtonian version of Raychauduri's equations. To achieve this, we use bootstrap argument and introduce new variables $\ke$, $\ks_{jk}$ and $\kb_{jk}$ (see \eqref{e:Wrs1}--\eqref{e:Wrs3}), the perturbations of the approximation solutions. However, directly analyzing $\ks_{jk}$ can not complete the bootstrap. Instead, we use $\kss:=\ks^{jk}\ks_{jk}$ to analyze the perturbation of the shear. Analyzing this system and suitably small initial data of these perturbations, with the help of the bootstrap assumptions to reach the bootstrap improvements, conclude the expansion, shear and rotation are bounded eventually, and in turn,  $|\del{i}w^j|<\infty$.  Furthermore, the strong blowup theorem follows.

\section{Blowups of large, expanding and irregularly-shaped molecular clouds}\label{s:MaCoj}

This section dedicates to the dynamical bounds on the spreading rate of the support of the density of the molecular clouds and the weak blowup theorem \ref{t:blupthm0}.  In fact, the main mechanics of blowups in this article is in the level of the \textit{virial theorem}. Due to the irregular shape of this cloud and its irregular distribution of the Newtonian potential, the rotations of cloud usually is inevitable (in contrast to the spherical symmetric cases). The motion of the cloud becomes quite complicate and it is very difficult to analyze the detailed movements. The intuitive idea to overcome this difficulty is based on a simple observation that the far-field (very large distances from the center of mass) regions of the gravitational field of the molecular cloud behaves like  the surrounding of the ``normal'' point mass if $\rrho\in \mathrm{R}^{1}([0,T)\times\Omega(t))$, that is $\rrho\in \mathrm{R}^{1}([0,T)\times\Omega(t))$ makes sure the gravity on the boundary satisfies the inverse-square law and this will be proved in Proposition \ref{t:dphix2} (\S\ref{s:estNewpot}).
By the virial theorem or Sideris' average value  $ H (t)$ given in \eqref{e:xi0} (see, for example,   \cite{Sideris1985,Makino1990}), we obtain Theorem \ref{t:bluplem} (\S\ref{s:criterion}), a criterion of blowups of classical solutions. This Lemma indicates that if the boundary of the cloud is expanding below a certain rate, then there are blowups during the development of the solution.

By using this criterion of blowups of classical solutions and the previous inverse-square gravity, we conclude the estimates of trajectory of free falling particles on the boundary (proposition \ref{t:spprbd} of  \S\ref{s:traj} and see the proof in \S\ref{s:stp1}) and these estimates of trajectories yields the expanding rate of the boundary is smaller than the blowup rate and further leads to the weak blowup theorem \ref{t:blupthm0} in \S\ref{s:wbt}.

In addition, we point out the weak blowup theorem \ref{t:blupthm0} in some sense partially answers the conjecture proposed by Makino \cite{Makino1992}  that is any tame
solution will become not tame after a finite time as well.

\subsection{Estimates of Newtonian gravity on the boundary}\label{s:estNewpot}

One way of studying the long time nonlinear behaviors of solutions begins with the ideal cases with various symmetries and then the cases deviating from ideal ones slightly. However, there is no decent solutions with enough symmetries for molecular clouds since the clouds are usually irregularly-shaped in the real world. For this purpose, to ensure the problem is researchable and the molecular clouds are as real as possible, our approach is to focus on cases excluding certain global solutions defined by Definition \ref{t:12class}, i.e., the first  class of global solutions (which will be investigated in a separate paper). The reason we exclude the first class of global solution is by excluding them, the gravity obeys the inverse square law, see below Proposition \ref{t:dphix2}, and this property, resembling the distribution of gravity around a point mass, allows us to obtain the estimates of that the spreading rate of the boundary is not fast enough to sustain the virial theorem for all time under certain initial data.

We call \textit{$\Omega(t)$ contains a neighbor of origin}, if there exists a constant $\iota\in(0,1)$, such that $B_\iota(0) \subset \Omega(t)$.
\begin{proposition}\label{t:dphix2}
	Suppose $ \Phi$ is the Newtonian potential  defined by \eqref{e:Newpot} and $T\in(0,\infty]$ is a constant,
	$\rrho \in C^1([0,T)\times \Rbb^3)\cap \mathrm{R}^{1}([0,T)\times\Omega(t),G_1)$ and
	$ \Omega(t)$ is precompact  for every $t\in[0,T)$ and contains a neighbor of origin, then
	\begin{equation*}
	|\del{i} \Phi(t,\bx)|\leq  \frac{G_1}{|\bx|^{2}}
	\end{equation*}
	for any $(t, \bx) \in [0, T) \times  \del{}\Omega(t) $.
\end{proposition}
\begin{proof}
	Recalling \eqref{e:Newpot}, we have
	\begin{equation*}
	 \Phi(t,\bx)=-\frac{1}{4\pi}\int_{\Omega(t)} \frac{\rrho(t,\by)}{|\bx-\by|}d^3 \by \quad \text{for} \quad \bx\in \Rbb^3.
	\end{equation*}
	Then, with the help of Lemma \ref{t:dphi} (see Appendix \ref{s:Nton}), we obtain
	\begin{equation}\label{e:dphi1}
	\del{j} \Phi(t,\bx) =   \frac{1}{4\pi} \int_{\Omega(t)} \rrho (t, \by) \frac{  \delta_{kj}(x^k-y^k)}{|\bx-\by|^3}d^3\by
	\end{equation}
	for any $(t, \bx) \in [0,T) \times   \Rbb^3 $.

Let us estimate $|\del{i}  \Phi|$. Since $\rrho \in  \mathrm{R}^{1}([0, T)\times\Omega(t))$, by Definition \ref{t:diffdis}, there are constants  $\delta\in(0,1)$, such that $G_{1}= \frac{1}{ 4\pi}\bigl(\frac{  M  }{\delta^2} + 3M\bigr)$ and
for any  $t\in[0,T)$, every vector $\bx \in\del{}\Omega(t)$ and $\mathbf{r}\in B(\mathbf{n},\delta)$ where  $\mathbf{n}:=\bx/|\bx|$, let us calculate\footnote{Note the fact that for $k<3$ and $R>0$ first by letting $ \omega:= \by-\bx $ and applying the spherical coordinate $\omega^1=r\sin \theta \cos \phi$, $\omega^2=r\sin \theta \sin \phi$ and $\omega^3=r\cos \theta$,
	\begin{align}\label{e:spintg}
		\int_{B(\bx, R)}   \frac{1}{|\bx -\by  |^k} d^3 \by  =&\int_{|\omega|<R }   \frac{1}{|\omega |^k} d^3 \omega = 2\pi \int^\pi_0   \int^R_0   \frac{1}{r^k} r^2\sin  \theta dr d\theta = \frac{4\pi R^{3-k}}{3-k}.
\end{align}	}
\begin{align*}\label{e:regdef}
\int_{B(\mathbf{n},\delta)}\frac{ |\bx|^{3} \rrho(t,|\bx|\br )}{|\mathbf{n}-\mathbf{r}|^{2}}d^3\mathbf{r}<3M.
\end{align*}	
Then, using above bound, we arrive at 	for any $(t, \bx) \in [0,T) \times  \del{}\Omega(t) $,
\begin{align*}
	\frac{1}{4\pi} \int_{\Omega(t)\cap B(\bx,\delta|\bx|)}  \frac{ \rrho (t, \by)}{|\bx -\by |^2}d^3\by \overset{\by:=|\bx|\br}{\longeq} & \frac{1}{4\pi} \int_{B(\mathbf{n},\delta)\cap \{\br\;|\;|\bx|\br\in\Omega(t) \}}  \frac{\rrho (t, |\bx|\br)|\bx|^3 }{\bigl|\bx -|\bx|\br \bigr|^2} d^3\br  \notag  \\
	\leq &\frac{1}{4\pi|\bx|^2} \int_{B(\mathbf{n},\delta)}  \frac{|\bx|^3 \rrho (t, |\bx|\br) }{|\mathbf{n} -\br |^2} d^3\br
	<  \frac{3M}{4\pi |\bx|^2}.
\end{align*}

Note for every $\bx\in \del{}\Omega(t)$ and $\by \in \Omega(t)\setminus B(\bx,\delta|\bx|)$, we obtain $|\bx-\by|\geq \delta|\bx|$. Further, for every $(t,\bx)\in [0,T)\times \del{}\Omega(t)$, we have
\begin{align*}
|\del{j} \Phi(t,\bx)|\leq  & \frac{1}{4\pi} \Bigl|\int_{\Omega(t)\setminus B(\bx,\delta|\bx|)}  \rrho (t, \by) \frac{  \delta_{kj}(x^k-y^k)}{|\bx -\by |^3}d^3\by\Bigr|  \notag  \\
& +\frac{1}{4\pi} \Bigl|  \int_{\Omega(t)\cap B(\bx,\delta|\bx|)}  \rrho (t, \by) \frac{  \delta_{kj}(x^k-y^k)}{|\bx -\by |^3}d^3\by  \Bigr|  \notag \\ \leq & \frac{1}{4\pi} \int_{\Omega(t)\setminus B(\bx,\delta|\bx|)}   \frac{\rrho (t, \by)}{|\bx -\by |^2}d^3\by  + \frac{1}{4\pi} \int_{\Omega(t)\cap B(\bx,\delta|\bx|)}  \frac{\rrho (t, \by)}{|\bx -\by |^2}d^3\by \notag \\
\leq & \frac{1}{4\pi\delta^2}\frac{1}{|\bx |^2} \int_{\Rbb^3}  \rrho (t, \by) d^3\by+ \frac{3M}{ 4\pi |\bx|^2} v
\leq v\Bigl( \frac{  M  }{\delta^2} + 3M\Bigr) \frac{1}{4\pi|\bx|^2}<  \frac{G_1}{|\bx|^2}.
\end{align*}
Then we complete this proof.
\end{proof}

\subsection{Conservations of velocity of center of mass, mass and energy}\label{s:consv}
In order to simplify the calculations, we take a specific coordinate, the \textit{center of mass frame}\footnote{The \textit{center of mass frame} is an \textit{inertial} frame in which the center of mass of a system is \textit{at rest} with respect to the origin of the coordinate system.}, and let the center of mass locate at the origin of the coordinate system. The following Lemma \ref{t:ctoms} guarantees such a coordinate system is inertial and thus exists, and this lemma also indicates that the mass and total energy are conserved quantities.
Let us first define the \textit{center of mass} by
\begin{equation}\label{e:ctoms}
x^i_c:=x^i_c(t)=\frac{1}{ M } \int_{\Rbb^3} \rrho x^i d^3\bx
\end{equation}
where we recall $ M $, $E$ are the \textit{total mass}
and \textit{total energy} given by \eqref{e:mass} and \eqref{e:eng0}.

Let us first present Lemma \ref{t:iden} which provides two useful identities for the conservation laws in Lemma \ref{t:ctoms} and the later proof. We state the following Lemmas \ref{t:iden} and \ref{t:ctoms} here with the proof given in Appendix \ref{a:cons}.
\begin{lemma}\label{t:iden}
	There are following identities for $\rrho \in C^1([0,\infty)\times \Rbb^3)$ with the precompact support $\Omega(t)$ and $ \Phi(t,\bx)$ defined by \eqref{e:Newpot},
	\begin{gather}
	\int_{\Omega(t)}   \bigl( \rrho(t,\bx) \del{ }^k \Phi (t,\bx)  \bigr) d^3\bx =0  \label{e:rphin0}  \intertext{and}
	\int_{\Omega(t)}  \rrho(t,\bx)  \del{j} \Phi(t,\bx)  x^j d^3\bx = -\frac{1}{2}\int_{\Omega(t)}  \rrho (t,\bx)   \Phi(t,\bx)  d^3\bx. \label{e:rpx3}
	\end{gather}
\end{lemma}
Then, the conservation laws are given in the following lemma,
\begin{lemma}\label{t:ctoms}
	Suppose $(\rrho, \mathring{w}^i)$ solves the Euler--Poisson equations \eqref{e:NEul1}--\eqref{e:inidata} on $[0,T)\times\Omega(t)$, the center of mass $x^i_c$, the total mass $ M $ and the total energy $ E $ are defined by \eqref{e:ctoms}--\eqref{e:eng0}. Then this velocity of the center of mass $\frac{d x^k_c(t)}{dt}$, the total mass $ M (t)$ and the total energy $ E (t)$ are conserved.
	Moreover, the velocity of the center of mass can be expressed by
	\begin{equation}\label{e:vlcm}
	\frac{d x^k_c(t)}{dt}=\frac{1}{ M } \int_{\Rbb^3} \rrho\mathring{w}^k d^3\bx=\text{const.}
	\end{equation}
for $t\in[0,T)$.
\end{lemma}

Lemma \ref{t:ctoms} demonstrates that the center of mass of this molecular cloud is moving in a constant velocity \eqref{e:vlcm}. Therefore, we are able to set the center of mass to be the origin of the coordinate system and make sure this system is inertial.
From now on, we set the center of mass $x^i_c=0$ (i.e. the origin is situated at the center of the mass of the molecular cloud), and therefore, the definition \eqref{e:ctoms} implies
\begin{align*}
\int_{\Rbb^3}\rrho x^i d^3\bx=0.
\end{align*}

\subsection{Criterion of blowups of classical solutions and virial theorem}\label{s:criterion}

In this section, we are in the position to state the evolution theorem of classical solutions of the diffuse boundary problem. This theorem is closely related to the Makino's conjecture in \cite{Makino1992} that states \textit{any tame
solution, including nonsymmetric tame solution, will become not tame after a finite time}. The Theorem \ref{t:blupthm0} presented later in fact answers Makino's conjecture in a perspective for classical solutions of the diffuse boundary problem.
We firstly give a lemma stating if the expanding rate of the clouds is close to certain linear expanding rate, then the solution can only be extended to finite time (controlling the expanding rate implies blowups).
This Theorem \ref{t:bluplem} is in fact one expression of the virial theorem. The idea of Theorem \ref{t:bluplem} comes from Sideris \cite{Sideris1985} and then is applied by many authors (for instance, \cite{Brauer1998,Liu1996,Pan2005}, etc.).

\begin{theorem}[Criterion of Blowups, Virial Theorem]\label{t:bluplem}
	Suppose $(\rrho,\rw^i,\Omega(t))$ is a classical solution for $t\in [0,T) $ to Euler--Poisson equations \eqref{e:NEul1}--\eqref{e:inidata}, $\beta :=\min\{3(\gamma-1),1\}>0$, $ M >0$ and $ E $ are mass and energy defined by \eqref{e:mass} and \eqref{e:eng0}, respectively, the total energy $E$ of this system is positive\footnote{It means the averaging trend of the clouds is expansion, since this, by \eqref{e:xiecomp1} in the following proof, leads to the time derivative of virial is positive, i.e., $ H^{\prime\prime} (t)>0$. See \cite{Krumholz2017} for more explanations. }, and for every $t\in[0,T)$, if there is a scalar function $R(t)\in C^0([0,T))$ such that the support of the density $\Omega(t)=\supp\rrho(t,\cdot) \subset B_{R(t)}(0)$. Then,
	\begin{enumerate}
		\item if $T=\infty$, then
		$R(t)$ satisfies
		\begin{equation}\label{e:limcri0}
			\lim_{t\rightarrow T}  \frac{R^2(t)}{t^2}  \geq \frac{\beta  E }{ M } .
		\end{equation}
	Otherwise, if \eqref{e:limcri0} fails and denote $[0,T_\text{max})$ is a maximal interval of existence of this solution, then $T_\text{max}<\infty$ is finite.
		\item if $T<\infty$, then
		 $R(t)$ satisfies
		\begin{equation}\label{e:limcri}
			\lim_{t\rightarrow T}  R^2(t)  \geq  \frac{\beta  E T^2}{ M }+\frac{2H^\prime(0) T}{M }+\frac{2H(0)}{M},
		\end{equation}
	Otherwise, if \eqref{e:limcri} fails, then $T_\text{max}<T<\infty$.
	\end{enumerate}
\end{theorem}

\begin{proof}
We define the moment of inertia for $t\in[0,T)$ by 
	\begin{align}\label{e:xi0}
		 H (t):=\frac{1}{2}\int_{\Rbb^3}\rrho(t,\bx) |\bx|^2 d^3 \bx. 
	\end{align}

	On one hand, differentiate $ H (t)$ with respect to $t$ by using Reynolds transport Theorem \ref{t:RTT}, and then the boundary term vanishes due to the fact that $\rrho\equiv 0$ on $\del{}\Omega(t)$,
	we have
	\begin{align}\label{e:dtH}
		\frac{d}{dt} H (t)=&\frac{1}{2}\int_{\Omega(t)}\del{t}\rrho |\bx|^2 d^3\bx
		= -\frac{1}{2}\int_{\Omega(t)} \del{i}( \mathring{w}^i  \rrho)  |\bx|^2 d^3\bx
		=   \int_{\Omega(t)}    \rrho  \mathring{w}^i x^j  \delta_{ij} d^3\bx
	\end{align}
	for $t\in[0,T)$. This quantity $\int_{\Omega(t)}    \rrho  \mathring{w}^i x^j  d^3\bx$ is also known as \textit{virial tensor} in physics.
	
	Then differentiate above $\frac{d}{dt} H (t)$ again with respect to $t$, with the help of \eqref{e:smeuler} and \eqref{e:rpx3} from Lemma \ref{t:iden}, we arrive at one expression of the famous \textit{Virial Theorem},
	\begin{align}\label{e:xiecomp1}
		\frac{d^2}{dt^2}  H (t) = & \int_{\Omega(t)}  \del{t}(  \rrho  \mathring{w}^k) x^j  \delta_{kj} d^3\bx    
		=   \int_{\Omega(t)}  \bigl[-\del{i} (\rrho \mathring{w}^i\mathring{w}^k)-\delta^{ik} \del{i}  p   - \rrho \del{}^k \Phi\bigr] x^j  \delta_{kj} d^3\bx   \notag  \\	
		&\hspace{-1.5cm} =   \int_{\Omega(t)}  -\del{i} (\rrho \mathring{w}^i\mathring{w}^k)x^j  \delta_{kj} d^3\bx-\int_{\Omega(t)} \delta^{ik} \del{i}  p  x^j  \delta_{kj} d^3\bx -\int_{\Omega(t)}  \rrho \del{j} \Phi  x^j d^3\bx   \notag  \\	
		&\hspace{-1.5cm} =  \int_{\Omega(t)}   \rrho \mathring{w}^i\mathring{w}^j    \delta_{ij} d^3\bx + 3\int_{\Omega(t)}    p  d^3\bx + \frac{1}{2}\int_{\Omega(t)}  \rrho   \Phi d^3\bx   \geq \beta  E  >0
	\end{align}
	for $t\in[0,T )$
	where $\beta :=\min\{3(\gamma-1), 1\}>0$.
	Then, integrating \eqref{e:xiecomp1} yields,	for $t\in[0,T )$,
	\begin{align}\label{e:xilow}
		 H  (t) \geq \frac{1}{2}\beta E  t^2 + H  ^\prime (0) t+ H  (0). 
	\end{align}

	On the other hand, $\bx\in\Omega(t)\subset B_{R(t)}(0)$ and \eqref{e:xi0} leads to, 
	for $t\in[0,T )$, 
	\begin{align}\label{e:xiupp}
		 H (t) =\frac{1}{2}\int_{\Rbb^3}\rrho |\bx|^2 d^3\bx=\frac{1}{2}\int_{\Omega(t)}\rrho |\bx|^2 d^3\bx \leq \frac{1}{2}R^2(t)\int_{\Rbb^3}\rrho  d^3\bx=\frac{1}{2}R^2(t) M. 
	\end{align}
	
	Combining \eqref{e:xilow} with \eqref{e:xiupp} results that
	\begin{equation}\label{e:ctreq}
		F(t):=\frac{1}{2} \bigl[\beta E  t^2 -   M  R^2(t)\bigr] + H  ^\prime (0) t+ H  (0) \leq 0,
	\end{equation}
	for any $t\in[0,T )$.
	Note that $F(t)$ is a continuous function in $[0,T )$ and directly, \eqref{e:xiupp} implies, initially,
	\begin{equation}\label{e:f0}
		F(0)= H  (0)-  \frac{1}{2}  M  R^2(0) \leq 0.
	\end{equation}
$ E (0)>0$, $\beta>0$ and  \eqref{e:ctreq} yield
	\begin{equation}\label{e:limF}
		\lim_{ t \rightarrow T} \frac{2F(t)}{ M  R^2(t)}
		= 	\lim_{ t\rightarrow T } \Bigl(\frac{\beta E  t^2}{ M R^2(t)}+\frac{2 H ^\prime(0) t}{ M  t^2}\frac{t^2}{R^2(t)} +\frac{2 H (0)}{ M t^2}\frac{t^2}{ R^2(t)}\Bigr)-1 \leq 0.  
	\end{equation}
If $T=\infty$, \eqref{e:limF} implies \eqref{e:limcri0}, and if $T<\infty$, leads to \eqref{e:limcri}.

On the contrary, for $T<\infty$ and
\begin{equation*}
	\lim_{t\rightarrow T}  R^2(t)  < \frac{\beta  E T^2}{ M }+\frac{2H^\prime(0) T}{M }+\frac{2H(0)}{M},
\end{equation*}
by \eqref{e:limF}, we arrive at $\lim_{t\rightarrow T}F(t)>0$  (it is similar for $T=\infty$ to obtain $\lim_{t\rightarrow T}F(t)>0$). With the help of \eqref{e:f0} and the intermediate value theorem, there exists a time $T_0>0$ such that for $t\geq T_0$, $F(t)> 0$, which contradicts with the inequality \eqref{e:ctreq} for $t\in[0,T )$. This contradiction concludes Theorem \ref{t:bluplem}.
\end{proof}

If $R(t):=2A (t+a)$,  we obtain the following Corollary \ref{e:supcri}, and this corollary will be used in the next subsection \S\ref{s:traj}.
\begin{corollary}\label{e:supcri}
	Under assumptions of Theorem \ref{t:bluplem}, if further assuming
	\begin{align*}
		R(t):=2A (t+a)
	\end{align*}
where $A<\frac{1}{24}\sqrt{\frac{\beta E}{M}}$,
then there is a critical time, $T^\dag>0$ given by
\begin{align*}
	T^\dag:=\frac{4aM A^2-H^\prime(0)+\sqrt{(H^\prime(0)-4aMA^2)^2-2(\beta E-4MA^2)(H(0)-2MA^2a^2)}}{\beta E- 4 A^2 M}
\end{align*}
such that $F(t)>0$ for $t>T^\dag>0$. Furthermore, the maximal interval of the existence of the solution $T_\text{max} \leq T^\dag$.
In specific, if $a:=\sigma^{-1}$ ($\sigma\in(0,\sigma_\star)$ and $\sigma^\star$ is defined by \eqref{e:sigma0}), there is a supercritical time $
	T^\natural:=\lambda_1 a+\lambda_2 \in(T^\dag,\frac{1}{9}a)$
where
\begin{align*}
	\lambda_1:= \frac{1}{10}	\AND \lambda_2:= \frac{9|H^\prime(0)|}{4 \beta E}<\frac{1}{200}a,
\end{align*}
such that $F(t)>0$ for $t>T^\natural>0$.
\end{corollary}
\begin{proof}
	By solving inequality
	\begin{align*}
		F(t)=\frac{1}{2} \bigl[\beta E  t^2 -   4M  A^2 (t+a)^2\bigr] + H  ^\prime (0) t+ H  (0) >0,
	\end{align*}
and noting $A<\frac{1}{24}\sqrt{\frac{\beta E}{M}}$ and $F(0)\leq 0$ (see \eqref{e:f0}),
we can calculate the critical time $T^\dag>0$ such that for $t>T^\dag$, $F(t)>0$. From $F(t)>0$, we conclude, for any $T>T^\dag$,
\begin{align*}
	\lim_{t\rightarrow T}  R^2(t) = 4 A^2 (T+a)^2 < \frac{\beta  E T^2}{ M }+\frac{2H^\prime(0) T}{M }+\frac{2H(0)}{M}.
\end{align*}
Then by Theorem \ref{t:bluplem}, we arrive at $T_\text{max}<T$, i.e., $T_\text{max}\leq T^\dag$.

Next, we prove $T^\natural>T^\dag$. To achieve this, we first note that
	\begin{align}
		H(0)-2MA^2a^2 \leq &0 \label{e:ineq2}
	\end{align}
where \eqref{e:ineq2} comes from \eqref{e:f0}. Then we verify
	\begin{align}\label{e:rt1}
	&\sqrt{(H^\prime(0)-4aMA^2)^2-2(\beta E-4MA^2)(H(0)-2MA^2a^2)} \notag  \\
	\leq &  |4aMA^2-H^\prime(0)|+\sqrt{2(\beta E-4MA^2)(2MA^2a^2-H(0))}	\notag  \\
	\leq &  4aMA^2+|H^\prime(0)|+2A a\sqrt{M(\beta E-4MA^2)}. 	
	\end{align}
Since $A<\frac{1}{24}\sqrt{\frac{\beta E}{M}}$, we obtain
\begin{align}
	&\frac{8MA^2+2A\sqrt{(\beta E-4A^2 M)M}}{\beta E-4A^2 M}<\frac{2}{143}+\frac{1}{\sqrt{143}}<\frac{1}{10}=:\lambda_1, \label{e:rt2}
	\intertext{and}
	&\frac{2|H^\prime(0)|}{\beta E-4A^2 M}<\frac{288|H^\prime(0)|}{143 \beta E}< \frac{9|H^\prime(0)|}{4 \beta E}=:\lambda_2. \label{e:rt3}
\end{align}
Definition of $\sigma_\star$ \eqref{e:sigma0} implies $\lambda_2<\frac{1}{200}a^{-1}$.
Gathering \eqref{e:rt1}--\eqref{e:rt3} together, we obtain $T^\dag-T^\natural<0$, i.e., $T^\natural>T^\dag$ and it completes the proof.
\end{proof}

\subsection{Estimates of trajectories of free falling particles}\label{s:traj}

The subsequent Proposition \ref{t:spprbd} states the evolution of the boundary $\del{}\Omega(t)$ along the velocity field provided there is a classical solution,  especially expounds the dynamical estimates of the spreading rate of the boundary of the molecular clouds and the changes of the velocity field.

\begin{proposition}[Trajectories of Free Falling Particles]\label{t:spprbd}
	Suppose $T^\prime\in(0,\infty]$ is a constant,  $(\rrho,\rw^i,\Omega(t))$ is a classical solution for $t\in [0, T^\prime) $ to the diffuse problem of Euler--Poisson equations \eqref{e:NEul1}--\eqref{e:ffbdry0}, defined by Definition \ref{t:clsl}, in the center of mass frame, $ \Phi$ is given by \eqref{e:Newpot}, $ \Omega(0)$ is precompact and contains a neighbor of origin, $\rrho \in C^1([0,T^\prime)\times \Rbb^3)\cap \mathrm{R}^1([0,T^\prime)\times\Omega(t),G_1)$, $G_1$ is the gravity bound defined in Definition \ref{t:diffdis} and the given total mass and energy satisfy (i.e., $(E,M,G_1)$ is compatible)
	\begin{equation*}
		(9G_1)^{\frac{1}{3}}<\frac{1}{24}\Bigl(\frac{\beta E}{M}\Bigr)^{\frac{1}{2}}.
	\end{equation*}
Then for
any given constant $A \in \bigl((9G_1)^{\frac{1}{3}}, \frac{1}{24} \sqrt{\frac{\beta E}{ M}} \bigr)$ and any small constant $\sigma \in (0,\sigma_\star)$ where the constant $\sigma_\star$ is given by \eqref{e:sigma0}, if
	$\xi\in \del{}\Omega(0)$ lies in  \begin{equation}\label{e:ddball}
\frac{ 2  }{2- \sigma} A\sigma^{-1}<|\chi(0,\xi)|=|\xi| < \frac{1}{1-\sigma} \Bigl(1+ \frac{1}{14}\sigma\Bigr) A\sigma^{-1},
	\end{equation}
(recall \S\ref{s:lag} and see \cite[Theorem $2.15$]{Liu2021d} for the definition of $\chi$)  and the initial boundary velocity $\mathbf{w}_0|_{\del{}\Omega(0)} $ satisfies
	\begin{gather}
	 \frac{1-\sigma}{ A}  \sigma^{2}|\xi|^2 <z_0(\xi)  < \Bigl(1+ \frac{1}{14}\sigma\Bigr) \sigma |\xi|  
	 \AND
	 X_0^2(\xi) < \frac{A}{2}\sigma^{3}|\xi|.  \label{e:data2c1a}
	\end{gather}
Then, there are, by denoting  $a:=\sigma^{-1}$, estimates of the flow and velocities,
\begin{gather*}
	A (t+a)<|\chi(t,\xi) |   < \frac{A}{1-2\sigma}(t+a)\Bigl[1+ \frac{1}{7} (t+a)^{-1}\Bigr], \\
	\frac{1-2\sigma}{ A}(t+a)^{-2}|\chi(t,\xi) |^2<z(t,\xi)  <(t+a)^{-1}\Bigl[1+\frac{1}{7}(t+a)^{-1}\Bigr]|\chi(t,\xi) | , \\
	X^2(t,\xi)  <  A \sigma^{ 2} (t+a)^{-1 }|\chi(t,\xi) |.
\end{gather*}
	for any $(t,\xi)\in[0,\min\{T^\natural,T^\prime\})\times\del{}\Omega(0)$ where $T^\natural$ is given by Corollary \ref{e:supcri}.
\end{proposition}

\begin{remark}\label{t:spprbd2}
	In terms of admissible parameters given in Definition \ref{t:paramtr}.\eqref{t:Def1}, \eqref{e:ddball} and \eqref{e:data2c1a} are expressed by $|\xi|=\lambda_0 A \sigma^{-1}$ and $z(0,\xi)= \lambda_1 \lambda_0^{-1}  \sigma |\xi|=\lambda_1 A $. This will be used in \S\ref{s:stp1}. 	
\end{remark}

\subsection{Weak blowup theorem}\label{s:wbt}
Let us now present the weak blowup theorem.

\begin{theorem}[Weak Blowup Theorem]  \label{t:blupthm0}
    Suppose  $(E,M,G_1)$ is any given compatible set and  $H^\prime(0)$ is any given constant, $(\rrho,\rw^i,\Omega(t))$ is a classical solution for $t\in [0, T) $ to the diffuse boundary problem of Euler--Poisson equations \eqref{e:NEul1}--\eqref{e:ffbdry0} for some constant $T\in(0,\infty]$ in the center of mass frame and it is not the first class of the global solution with  $G_1>0$. Then
if  the initial data $(\mathbf{w}_0|_{ \del{}\Omega(0)}, \Omega(0)) \in \mathfrak{A}(E,M,G_1,H^\prime(0))$ is admissible initial boundary data, then it is not the global solution to the diffuse boundary problem, and there is a constant $0<T^\star<+\infty$, such that, the classical solution breaks down at $t=T^\star$ and
	\begin{align}\label{e:blupest1}
		 \int^{T^\star}_0 \Bigl(\|\nabla w^i(s)\|_{L^\infty( \Omega(s))}  +\|\nabla \rrho^{\frac{\gamma-1}{2}}\|_{ L^\infty(\Omega(s))} \Bigr) ds=+\infty;
	\end{align}
\end{theorem}
This theorem provides a necessary condition for global classical solutions to the diffuse boundary problem of Euler--Poisson equations. We state it in the following corollary and its proof is direct from Theorem \ref{t:blupthm0}.
\begin{corollary}\label{t:glsl}
	Under the assumptions of Theorem \ref{t:blupthm0}, if $(\rrho,\rw^i,\Omega(t))$ is a global solution to the diffuse boundary problem of Euler--Poisson equations \eqref{e:NEul1}--\eqref{e:ffbdry0} for $t\in [0, +\infty) $, then it has to be the first class of global solution with some gravity bound $G_1>0$ and includes near-boundary mass accretions.
\end{corollary}

\subsection{Proof of Theorem \ref{t:blupthm0}}\label{s:stp1}
This theorem is a result of the virial theorem (see Theorem \ref{t:bluplem}), Proposition \ref{t:spprbd} and the weak continuation principle (see \cite[\S$3.1$]{Liu2021d}). We first admit Proposition \ref{t:spprbd} holding in this proof, and defer the proof of Proposition \ref{t:spprbd} until the next section \S\ref{s:bdyevo}. The idea of this proof is for every initial boundary point, we verify the admissible initial boundary data satisfy the condition of Proposition \ref{t:spprbd}. Since there is no the first class of global solution, we can use Proposition \ref{t:spprbd} to conclude the spreading speed of the boundary is bounded by a linear function of time. Then the virial theorem (see Theorem \ref{t:bluplem}) can be applied due to this spreading speed of the  boundary and $E(0)>0$. At the end, we use the weak continuation principle  to obtain the blowup estimate \eqref{e:blupest1}. Since otherwise, the left hand of \eqref{e:blupest1} is bounded for every $T^\star<T$ where $T$ is the maximal time interval of the existence of the solution. Then the weak continuation principle  yields  there is a constant $T^\flat>T>0$, such that the classical solution $(\rrho,\rw^i,\rPhi,\Omega(t))$ exists on $t\in[0,T^\flat)$ which contradicts the fact that $T$ is the maximal time interval of the existence of the solution.

\begin{proof}[Proof of Theorem \ref{t:blupthm0}]
	We prove it by contradiction. Along with the assumptions of this theorem, we suppose this solution is a global classical solution to the Euler--Poisson equations \eqref{e:NEul1}--\eqref{e:inidata} on $[0,+\infty)\times \Omega(t)$ extending  $(\rrho, \mathring{w}^i,\Omega(t))$  to  $T_\text{max}=+\infty$ and  since it is not the first class of global solution with $G_1$, by Definition \ref{t:12class} we have $\rrho \in \mathrm{R}^1([0,+\infty)\times\Omega(t),G_1)$. In order to use Proposition \ref{t:spprbd} and in view of Remark \ref{t:spprbd2} and the initial data satisfies $(\mathbf{w}_0|_{ \del{}\Omega(0)}, \Omega(0)) \in \mathfrak{A}(E,M,G_1,H^\prime(0))$, for any $\xi\in\del{}\Omega(0)$ and any given radial velocity field $z_0(\xi)$, there is a small constant $\sigma_\xi\in(0,\sigma_\star)$ where $\sigma_\star$ is given by \eqref{e:sigma0}, such that $\sigma_\xi$ satisfies $\lambda_1\lambda_0^{-1}\sigma_\xi =z_0(\xi)/|\xi|$ and the data satisfy Definition \ref{t:paramtr}.\eqref{t:Def3}. Let us define\footnote{The freedoms of constants $A$ and $\sigma$ in Proposition \ref{t:spprbd} become the freedoms of the initial data $z_0$ and $\xi$ by \eqref{e:Az}. },
	\begin{equation}\label{e:Az}
		A_\xi:=\frac{1}{\lambda_1}z_0(\xi).  
	\end{equation}
	Then  for every $z_0$, by \eqref{e:xidn0},
	\begin{align}\label{e:z0}
		z_0(\xi) \in \biggl(\lambda_1(9G_1)^{\frac{1}{3}}, \frac{\lambda_1}{24} \sqrt{ \frac{ \beta E }{ M} }\biggr)
	\end{align}
	By above \eqref{e:Az} and  \eqref{e:z0}, $A_\xi\in\bigl((9G_1)^{\frac{1}{3}}, \frac{1}{24} \sqrt{\frac{\beta E}{M}} \bigr)$ are any $\xi$-dependent constants, and $0<\sigma_\xi<\sigma_\star $.
	These, with the help of interval of $\lambda_0$ and $|\xi|=\frac{z_0(\xi)}{\lambda_1\lambda_0^{-1}\sigma_\xi}=\lambda_0\sigma_\xi^{-1}A_\xi$, imply $\xi$ lies in
	\begin{equation*}
		\frac{ 2  }{2- \sigma} A_\xi\sigma_\xi^{-1}< |\xi| < \frac{1}{1-\sigma_\xi} \Bigl(1+ \frac{1}{14}\sigma_\xi\Bigr) A_\xi\sigma_\xi^{-1}.
	\end{equation*}
	Using \eqref{e:Az} again yields
	\begin{equation}\label{e:z0cal}
		z_0(\xi)=\lambda_1\lambda_0^{-1}\sigma_\xi |\xi|.
	\end{equation}	
	Then by the interval of $\lambda_1$ (see Definition \ref{t:paramtr}.\eqref{t:Def1}) and noting $\lambda_0=\frac{\lambda_1\sigma_\xi|\xi|}{z_0(\xi)}=\frac{\sigma_\xi|\xi|}{A_\xi}$, we obtain
	\begin{align*}
		\frac{1-\sigma_\xi}{ A}  \sigma_\xi^{2}|\xi|^2 <z_0(\xi)  < \Bigl(1+ \frac{1}{14}\sigma_\xi\Bigr) \sigma_\xi |\xi|.
	\end{align*}
	By \eqref{e:Xdn}, \eqref{e:Az} and \eqref{e:z0cal}, we obtain
	\begin{equation*}
		X_0^2(\xi) < \frac{A_\xi}{2}\sigma_\xi^{3}|\xi|.
	\end{equation*}
Then, we can apply Proposition \ref{t:spprbd} to conclude
	\begin{align}\label{e:xest}
		A_\xi (t+\sigma_\xi^{-1})<|\chi(t,\xi) |   < \frac{A_\xi}{1-2\sigma_\xi}(t+\sigma_\xi^{-1})\Bigl[1+ \frac{1}{7} (t+\sigma_\xi^{-1})^{-1}\Bigr]		
	\end{align}
for $(t, \xi )\in [0,T^\natural)\times \del{}\Omega(0)$. Since $A_\xi$ and $\sigma_\xi$ in \eqref{e:Az} are both continuous in $\xi$, and $\del{}\Omega(0)$ is a compact set, then there are uniform constants $A:=\max_{\xi\in\del{}\Omega(0)} A_\xi$ and $\sigma:=\min_{\xi\in\del{}\Omega(0)} \sigma_\xi$. We conclude
	\begin{align*}
|\bx|=|\chi(t,\xi)|< \frac{A}{1-2\sigma}(t+\sigma^{-1})\Bigl[1+ \frac{1}{7} (t+\sigma^{-1})^{-1}\Bigr]<2A(t+\sigma^{-1})=:R(t).
\end{align*}
for $t\in [0,T^\natural)$.

Then, by Corollary \ref{e:supcri}, we conclude $T_\text{max}\leq T^\dag<T^\natural$ which contradict with $T_\text{max}=+\infty$. Therefore, there is no global classical solution. In other words, if there is a global solution to the diffuse boundary problem of Euler--Poisson equations \eqref{e:NEul1}--\eqref{e:ffbdry0} for $t\in [0, +\infty) $, then $\rrho \notin   \mathrm{R}^1([0,+\infty)\times\Omega(t),G_1)$. By the weak continuation principle, we obtain the estimate \eqref{e:blupest1}. Then we complete this proof.
\end{proof}

\subsection{Proof of Proposition \ref{t:spprbd}}  \label{s:bdyevo}

We prove Proposition \ref{t:spprbd} by the following two steps. The first step focus on the reformulation of equations on the boundary, and the second step is the analysis of these equations and try to estimate the solutions to a  certain large time.

\subsubsection{Step $1$: reformulation of free falling equation on the boundary in Lagrangian description}
We begin with deriving a system describing the behavior of the extended velocity fields and positions of points on $[0, T^\prime) \times\del{}\Omega(t)$. For any $\bx\in \del{}\Omega(t)$, since $\chi$ is a flow, then there is a point $\xi\in \del{}\Omega(0)$, such that $\bx=\chi(t,\xi)$ with $\chi(0,\xi)=\xi$, $\xi$ is the initial position and $\chi(t,\xi)$ the Lagrangian variable (or flow). By \cite[Proposition $2.1.(3)$]{Liu2021d}, the extended velocity $w^k$ (see Definition \ref{t:clsl}.\eqref{D:1}) satisfies the equation of the free falling,
\begin{equation}\label{e:otvel2}
D_t w^i(t,\bx) =- \del{}^i \Phi(t,\bx), \quad\text{i.e.,}\quad  \del{t}\underline{w}^k(t,\xi)=-\underline{\del{}^k\Phi}(t,\xi).
\end{equation}
for $(t,\bx)\in [0,T^\prime)\times \del{} \Omega(t) $ and $(t,\xi)\in [0,T^\prime)\times \del{} \Omega(0) $. The free falling equation \eqref{e:otvel2} coupling with the definition of the integral curve $\chi^\xi(t)$ ($=\chi(t,\xi)$) of  $\mathring{\mathbf{w}}$, (see \cite[eq. $(2.14)$]{Liu2021d} by denoting $x^i=\chi^i(t,\xi)$ is the $i$-th component of $\bx=\chi(t,\xi)$) yield
\begin{align}
	\del{t}\chi^i(t,\xi)=& \underline{w}^i(t,\xi),  \label{e:lag1}\\
	\del{t} \underline{w}^i(t,\xi) =&- \delta^{ij}\Phi_j(t,\xi).  \label{e:lag2}
\end{align}
for $(t,\xi)\in [0,T^\prime)\times \del{} \Omega(0) $, where we denote $\Phi_j(t,\xi):=\del{j} \Phi(t,\chi(t,\xi))=\underline{\del{j} \Phi}(t,\xi)$.
For the purpose of the analysis, we introduce new variables $q$, $z$ and $X$ defined by (recall \eqref{e:zdef}--\eqref{e:Xdef})
\begin{align}
q(t,\xi):=&\frac{1}{|\bx|}=\frac{1}{|\chi(t,\xi)|},  \label{e:nvar0a} \\ z(t,\xi):=& \delta_{ij}  \frac{\chi^j(t,\xi)}{|\chi(t,\xi)|}  \underline{w}^i(t,\xi), \label{e:nvar0b}  \\
X^i(t,\xi):=& \underline{w}^i(t,\xi)-\frac{\chi^i(t,\xi)}{|\chi(t,\xi)|} z(t,\xi) \label{e:nvar0c} \intertext{and}
X(t,\xi) :=& \sqrt{\delta_{ij} X^i(t,\xi) X^j(t,\xi)}  \label{e:nvar0d}
\end{align}
where $q$ is the reciprocal of the radial distance $|\bx|=|\chi|$, $z$ is the radial component of velocity $ \underline{w}^i$, $X^i$ is a component of the direction orthogonal to the radial direction  and $X$ is the length of $X^i$.
We point out, by direct calculations, $\delta_{ij}X^i \chi^j=0$, and furthermore,
\begin{align}\label{e:nvar0e}
X(t,\xi)
= & \sqrt{\delta_{jk}  \underline{w}^j(t,\xi) \underline{w}^k(t,\xi) - z^2(t, \xi)}.
\end{align}
In order to obtain a proper structure, we also introduce a new variable $Y(t,\xi)$ defined by
\begin{equation}\label{e:Y}
	Y(t,\xi) :=  q(t,\xi)  X^2(t,\xi) .
\end{equation}

Firstly, let us derive the \textit{equation of $q$}. By noting $\bx=\chi(t,\xi)$ and multiplying $2 \delta_{ij} \chi^j  $ on the both sides of \eqref{e:lag1}, we arrive at
\begin{equation*} \label{e:ch2}
\del{t} |\chi (t, \xi )|^2=2 \delta_{ij} \chi^j(t, \xi )  \del{t} \chi^i (t, \xi ) =2 \delta_{ij}  \chi^j(t, \xi )   \underline{w}^i (t, \xi ).
\end{equation*}
Then we arrive at
\begin{align} \label{e:ch2a}
\del{t} |\chi (t, \xi ) | =   \delta_{ij}  \frac{\chi ^j  }{|\chi  |}   \underline{w}^i (t, \xi )=z(t,\xi ),
\end{align}
and we re-express above equation in terms of $q$ and $z$ defined in \eqref{e:nvar0a}--\eqref{e:nvar0b} to get the first equation of the target system,
	\begin{align}\label{e:ch2b}
	\boxed{
	\del{t} q =  -q^2 z
          }.
	\end{align}

Secondly, we derive the \textit{equation of $z$}.
Multiplying $\delta_{jk} \frac{\chi^j}{|\chi|}$ on the both sides of \eqref{e:lag2}, using Leibniz rule, \eqref{e:lag1} and \eqref{e:ch2a}, yields
\begin{align}\label{e:ch1a}
&\del{t}\Bigl(\delta_{jk} \frac{\chi^j(t,\xi)}{|\chi(t,\xi)|}  \underline{w}^k(t,\xi)\Bigr)  \notag  \\
= & \frac{1}{|\chi(t,\xi)|}\bigl[\delta_{jk}   \underline{w}^j(t,\xi)  \underline{w}^k(t,\xi) -  z^2(t,\xi)\bigr]-  \frac{\chi^j(t,\xi)}{|\chi(t,\xi)|} \Phi_j(t,\xi).
\end{align}	
Then, with the help of \eqref{e:nvar0e}, we obtain the second equation of the target system by expressing above equation \eqref{e:ch1a} in terms of $q$, $z$ and $Y$,
\begin{align}\label{e:ch1b}
\boxed{
\del{t} z  =  q  X^2-  \frac{\chi^j}{|\chi|} \Phi_j
=  Y   - \frac{\chi^j}{|\chi|} \Phi_j }.
\end{align}

The next one is the \textit{equation of $X$}. We first calculate, by \eqref{e:lag2}, that
\begin{align}\label{e:ch3}
\del{t}| \underline{w}^k(t,\xi)|^2 = 2\delta_{jk}  \underline{w}^j(t,\xi)\del{t} \underline{w}^k(t,\xi)  =  - 2  \underline{w}^j (t,\xi)   \Phi_j(t,\bx).
\end{align}
Then, definitions of $X$ and $X^j$ (see  \eqref{e:nvar0c}--\eqref{e:nvar0e}), with the help of equations \eqref{e:ch1b} and \eqref{e:ch3}, lead to
\begin{align}\label{e:ch3a}
\del{t} X^2=- 2z q X^2 - 2 X^j  \Phi_j.
\end{align}

By noting \eqref{e:Y}, i.e., $
X = \sqrt{q^{-1} Y }$ and using \eqref{e:ch2b} and \eqref{e:ch3a}, we obtain a equation of $Y$,
\begin{equation}\label{e:ch4b}
\boxed{\del{t}Y=-3z q Y-2 q^{\frac{1}{2}} Y^{\frac{1}{2}} \Bigl(\frac{X^j}{X}\Phi_j\Bigr) } .
\end{equation}

Gathering \eqref{e:ch2b}, \eqref{e:ch1b} and \eqref{e:ch4b} together, we arrive at the system of equations of $q$, $z$ and $Y$,
\begin{align}
\del{t} q = & -q^2 z, \label{e:dqeq}   \\
\del{t} z = & Y   - \frac{\chi^j}{|\chi|} \Phi_j, \label{e:dzeq} \\
\del{t}Y = &
-3z q Y - 2 q^{\frac{1}{2}} Y^{\frac{1}{2}} \Bigl(\frac{X^j}{X}\Phi_j\Bigr) ,  \label{e:dyeq}
\end{align}	
for $(t,\xi)\in [0,T^\prime)\times \del{} \Omega(0) $. Therefore, the aim turns to the estimates of $q$, $z$ and $Y$ for $(t,\xi)\in [0,T^\prime)\times\del{}\Omega(0)$ since these new variables contain enough information of the original variables $|\bx|=|\chi(t,\xi)|$ and $\rw(t,\chi(t,\xi))$.

\subsubsection{Step $2$: estimates of the flow}
As we have seen in \S\ref{s:stp1}, the proof of Theorem \ref{t:blupthm0}, we need the upper bound estimate of $|\chi|$ to use virial theorem. In order to obtain this upper bound of $|\chi|$, we judiciously choose new proper variables since the upper and lower bounds of $|\chi|$, $z$ and $Y$  usually have different changing rate, then integrate the derived system from \eqref{e:dqeq}--\eqref{e:dyeq} and estimate upper and lower bounds for the selected variables by bootstrap arguments (see Appendix \ref{a:btpr}). We introduce the following variables denoted by ``upper variable'' $\tilde{\cdot}$ and ``lower variable'' $\underaccent{\tilde}{\cdot}$, respectively since we will estimate the upper bound for the upper variables and the lower bound for the lower variables, respectively. In the end the complete estimates come from combining the upper and lower bounds with different decay rates together.

\underline{Variables:}
Let us begin this step with defining proper variables for later analysis. For any $\xi\in \del{}\Omega(0)$, and for any constants $A \in \bigl((9G_1)^{\frac{1}{3}}, \frac{1}{24} \sqrt{\frac{\beta E}{M}}\bigr)$ and  $\sigma \in (0,\sigma_\star)$, such that initial data satisfy  \eqref{e:ddball}--\eqref{e:data2c1a}. We define,
\begin{align}
    \dq: = &  A (t+  a)  q,
 \label{e:dnvar1}\\
    \hV:= &  (t+a)  - (t+a)^{2} zq , \label{e:dnvar2a} \\
	\hU:= & (t+a)^2 zq^2 , \label{e:dnvar3a}  \\
	\hy:= & (t+a) Y, \label{e:dnvar3}
\end{align}
for $(t,\xi)\in[0,\min\{T^\natural,T^\prime\})\times \del{}\Omega(0) $ along with the  requirement \eqref{e:ddball}.
Then, direct calculations by using \eqref{e:dnvar1}--\eqref{e:dnvar3} yields
	\begin{align}
q= & \frac{1}{A} (t+ a)^{-1} \dq,
\label{e:nvar1}\\
zq= &
 (t+a)^{-1 }\bigl[1-(t+a)^{- 1} \hV\bigr] ,
\label{e:nvar1a}\\
zq^2=&  (t+a)^{-2} \hU  , \label{e:nvar2a}\\
Y = &  (t+a)^{- 1 }\hy.  \label{e:nvar3}
\end{align}


\underline{Initial data:}
Initial data \eqref{e:ddball}--\eqref{e:data2c1a}, in terms of $\dq$, $\hV$, $\hU$, $\hy$ (with the help of \eqref{e:dnvar1}$|_{t=0}$--\eqref{e:dnvar3}$|_{t=0}$), are expressed
\begin{align}
\dq_0:=&\dq|_{t=0}<1-\frac{1}{2}\sigma    ,    \label{e:data1} \\
\hU_0:=&\hU|_{t=0} > \frac{1-\sigma}{ A}  ,     \label{e:data2}  \\
\hV_0:=&\hV|_{t=0} >  - \frac{1}{14} ,     \label{e:data4}  \\
\hy_0:=&\hy|_{t=0} < \frac{1}{2}  A\sigma^2  ,  \label{e:data3} 
\end{align}

The rest of the proof is proceeded by the \textit{bootstrap principle}, see Appendix \ref{a:btpr}. In order to apply it, we give the following bootstrap assumptions, and then establish the system of equations of above new variables. Furthermore, by integrating these equations, we improve the estimates of the bootstrap assumptions, which leads to the expected estimates by Proposition \ref{t:btpr} (i.e., abstract bootstrap principle), eventually.

\underline{Bootstrap assumptions:}
	We assume for every $t\in[0,T)$ where $T$ is any constant such that $T\in(0,\min\{T^\natural,T^\prime\})$, the variables satisfy the following estimates,
	\begin{align}
	\dq < &  1 ,
	 \label{e:btasp1} \\
    \hU > & \frac{1-2\sigma}{ A} ,     \label{e:btasp4a}  \\
    \hV > & - \frac{1}{7} ,
     \label{e:btasp4}  \\
    \hy < & A \sigma^2 .  \label{e:btasp3}
	\end{align}
Direct calculations yield the initial data \eqref{e:data1}--\eqref{e:data3} satisfy the bootstrap assumptions \eqref{e:btasp1}--\eqref{e:btasp3}, which implies Proposition \ref{t:btpr}.$(4)$ is verified.

After giving the bootstrap assumptions, we intend to rewrite \eqref{e:dqeq}--\eqref{e:dyeq} in terms of $\dq$, $\hV$, $\hU$ and $\hy$, and then estimate these variables by integrating these new  equations and using the bootstrap assumptions. In fact, if the estimates of $\dq$, $\hV$, $\hU$ and $\hy$ are improved comparing with the bootstrap assumptions, then by Proposition \ref{t:btpr}, the bootstrap principle, we can conclude that the improved estimates holds on $t\in[0,T)$ where $T$ is any constant such that $T\in(0,\min\{T^\natural,T^\prime\}]$.

For late use, we point that the lower bound of $q$ can be estimated by
\begin{align}\label{e:qup1}
	q=\frac{zq^2}{zq}= \frac{(t+a)^{-1}\hU}{1-(t+a)^{-1} \hV }.
\end{align}
Before verifying the improvements of these inequalities, let us first give two simple but useful identities and estimates which will be used repeatedly in the following calculations. By \eqref{e:nvar1}, \eqref{e:btasp4} and the bootstrap assumption \eqref{e:btasp1}--\eqref{e:btasp4a},
\begin{align}\label{e:1zq1}
	1-(t+a)zq= (t+a)^{-1} \hV  >- \frac{1}{7}  (t+a)^{-1}>- \frac{1}{7}  \sigma
\end{align}
for $t\in[0,T)$ and using \eqref{e:nvar1},  \eqref{e:nvar2a} and the bootstrap assumption \eqref{e:btasp1}--\eqref{e:btasp4a}, we derive,
\begin{align}\label{e:1zq2}
	&1-(t+a)zq=1-(t+a)\frac{zq^2}{q}=1- \frac{A\hU}{\dq}<1-1+2\sigma=2\sigma
\end{align}
for $t\in[0,T)$.

Then by Corollary \ref{e:supcri}, we estimate an integration for later use,
\begin{align}
	\int^{T^\natural}_0(s+a)^{-1} ds=&
	\ln\Bigl(1+\frac{T^\natural}{a}\Bigr)= \ln\Bigl(1+\lambda_1+ \lambda_2 \sigma \Bigr) <\frac{1}{10}+  \frac{9|H^\prime(0)|}{4 \beta E} \sigma <\frac{21}{200}.    \label{e:intg1}
\end{align}

Now, let us improve the bootstrap assumptions.

\underline{Improvement of $\dq$:}
Differentiating \eqref{e:dnvar1}, with the help of \eqref{e:ch2b}, \eqref{e:nvar1} and \eqref{e:1zq2}, yields
\begin{align}\label{e:dtq1}
	& \del{t} \dq =  A (t+ a)  \del{t} q + A q=  Aq[ 1 - (t+ a) zq ] \notag  \\
	& \hspace{2cm}
    =  (t+a)^{-1} \dq \Bigl(1-\frac{A\hU}{\dq}\Bigr) < 2\sigma (t+a)^{-1} \dq.
\end{align}
Integrating \eqref{e:dtq1} and using \eqref{e:intg1} yield
\begin{align}\label{e:dtq2}
	\dq(t,\xi) <  & \dq_0+2\sigma\int^t_0(s+a)^{-1} ds <1-\frac{1}{2}\sigma +2\sigma\int^{T^\natural}_0(s+a)^{-1} ds\notag  \\
	< &  1-\frac{1}{2}\sigma +2 \Bigl(\frac{1}{10}+  \frac{9|H^\prime(0)|}{4 \beta E} \sigma \Bigr)\sigma   < 1-\frac{1}{4} \sigma
\end{align}
for $t\in[0,T)$.

\underline{Improvement of $\hU$:}
Let us first note
\begin{align}\label{e:zq2}
	&zq^2=zq\cdot q =  \frac{1}{A} (t+ a)^{-2} \dq \bigl[1-(t+a)^{-  1 } \hV\bigr] \notag  \\
	&\hspace{1cm} <\frac{1}{A} (t+a)^{-2}\Bigl[1+ \frac{ 1 }{7}(t+a)^{- 1 }\Bigr] <  \frac{8}{7A} (t+a)^{-2}.
\end{align}
Differentiating \eqref{e:dnvar3a}, with the help of \eqref{e:dqeq} and \eqref{e:dzeq}, we arrive at
\begin{align*}
	\del{t}\hU = & 2(t+a)zq^2+(t+a)^2 q^2\del{t}z+2qz(t+a)^2 \del{t} q \notag  \\
	= & 2(t+a)zq^2\bigl[ 1 - (t+a) z q\bigr] +(t+a)^2 q^2\Bigl( Y  - \frac{\chi^j}{|\chi|} \Phi_j\Bigr)  .
\end{align*}

By using \eqref{e:1zq1} and  \eqref{e:zq2}, we obtain
\begin{align*}
	-\frac{16}{7A}(t+a)^{-1}<(-2)\bigl((t+a) zq^2\bigr)<0
	\AND
	(t+a) z q-1< \frac{1}{7}(t+a)^{-1},
\end{align*}
which, in turn, implies
\begin{align*}
	(-2)\bigl((t+a) zq^2\bigr)\bigl[ (t+a) z q-1\bigr]> \frac{1}{7}(t+a)^{-1}(-2)\bigl((t+a) zq^2\bigr) >  -\frac{16 }{49A} (t+a)^{-2}.
\end{align*}
Then we estimate $ 2(t+a)zq^2\bigl[ 1 - (t+a) z q\bigr] $,
\begin{align}\label{e:atem}
	 2(t+a)zq^2\bigl[ 1 - (t+a) z q\bigr]
	 = &   (-2)\bigl((t+a) zq^2\bigr)\bigl[ (t+a) z q-1\bigr]  
	 >  -\frac{16 }{49A} (t+a)^{-2}.
\end{align}

Then using \eqref{e:dnvar3}, \eqref{e:1zq1}, \eqref{e:nvar2a}, \eqref{e:atem}, \eqref{e:qup1} and $(t+a)^2 q^2  Y>0$ yield
\begin{align}\label{e:dtu1}
\del{t}\hU
> &  -\frac{16 }{49A} (t+a)^{-2}    - \frac{1}{A^4} (t+a)^{-2}  \dq^4 \Bigl| \frac{1}{q^2}\frac{\chi^j}{|\chi|} \Phi_j \Bigr|.
\end{align}
Integrating \eqref{e:dtu1}, with the help of $A>(9G_1)^{\frac{1}{3}}$, we obtain
\begin{align}\label{e:dtu2}
	\hU(t,\xi) > & \hU_0  - \int^t_0   \frac{16 }{49 A} (s+a)^{-2} ds  - \int^t_0 \frac{1}{A^4} (s+a)^{-2}  \dq^4  \Bigl| \frac{1}{q^2}\frac{\chi^j}{|\chi|} \Phi_j \Bigr| ds \notag  \\
	> & \frac{1-\sigma}{ A}  -  \frac{16 }{49A}  \int^t_0  (s+a)^{-2}  ds   -  \frac{G_1 }{ A^4   }  \int^t_0 (s+a)^{-2} ds\notag  \\
	> & \frac{1-\sigma}{ A} -  \frac{16 \sigma }{49 A }   - \frac{ G_1  }{  A^4 } \sigma  \notag  \\
    > &   \frac{1-\sigma}{ A} -  \frac{16 \sigma}{49 A }    - \frac{1 }{ 9 A  } \sigma  =\frac{1 }{ A} -  \frac{65 \sigma}{49 A }    - \frac{1  }{ 9 A  } \sigma  > \frac{1-\frac{3}{2}\sigma}{ A}
\end{align}
for $t\in[0,T)$.

\underline{Improvement of $\hV$:}
Now Let us turn to the equation of $\hV$. First direct calculations imply that
\begin{align}\label{e:iden1}
&   1 - 2(t+a)  zq +(t+a)^{2} z^2 q^2
= \bigl[1-(t+a)zq \bigr]^2   \geq 0
\end{align}
Differentiating \eqref{e:dnvar2a}, with the help of \eqref{e:nvar2a},  \eqref{e:dqeq}--\eqref{e:dzeq}, \eqref{e:Y}, \eqref{e:1zq2} and \eqref{e:iden1}, we arrive at
\begin{align}\label{e:dtv1}
	\del{t}\hV = &   1-2(t+a)  zq -(t+a)^{2}z \del{t} q -(t+a)^{2} q \del{t}z \notag  \\
	= &  1- 2(t+a) zq +(t+a)^{2} z^2 q^2  -(t+a)^{2} q \Bigl(Y- \frac{\chi^j}{|\chi|} \Phi_j\Bigr)  \notag  \\
	\geq &    -(t+a)^{2} q Y  + (t+a)^{2} q^3 \Bigl(\frac{1}{q^2}\frac{\chi^j}{|\chi|} \Phi_j\Bigr) \notag  \\
	> &  -\frac{ \dq\hy}{A}      - \frac{ \dq^3}{A^3}  (t+a)^{-1}   \Bigl|\frac{1}{q^2}\frac{\chi^j}{|\chi|} \Phi_j \Bigr|.
\end{align}

Integrating \eqref{e:dtv1},  noting Corollary \ref{e:supcri} and using $A>(9G_1)^{\frac{1}{3}}$ lead to
\begin{align}\label{e:dtv2}
	\hV(t,\xi) > & \hV_0  -\frac{1}{A}  \int^t_0   \dq\hy ds  - \frac{1}{A^3}  \int^t_0  (s+a)^{-1}\dq^3\Bigl|\frac{1}{q^2}\frac{\chi^j}{|\chi|} \Phi_j \Bigr| ds \notag \\
	> & - \frac{1}{14} -    \sigma^2  T^\natural  - \frac{G_1  }{  A^3 }  \int^t_0  (s+a)^{-1} ds\notag \\
	> & - \frac{1}{14}   -     \frac{1}{10}  \sigma -\frac{9|H^\prime(0)|}{4 \beta E} \sigma^2     - \frac{G_1}{  A^3   }  \Bigl(\frac{1}{10}+  \frac{9|H^\prime(0)|}{4 \beta E} \sigma \Bigr) \notag  \\
	>& -\frac{26}{315}  -     \frac{1}{10}  \sigma -\frac{9|H^\prime(0)|}{4 \beta E} \sigma^2     -   \frac{ |H^\prime(0)|}{4 \beta E} \sigma  >   -  \frac{3}{28}
\end{align}
for $t\in[0,T)$ by noting $\sigma\in(0,\sigma_\star)$.

\underline{Improvement of $\hy$:}
The following expressions of $\del{t}\hy$ can be derived by directly differentiating \eqref{e:dnvar3}. That is,  with the help of \eqref{e:dyeq}, \eqref{e:Y}, \eqref{e:nvar1}, \eqref{e:qup1} and the crucial identity \eqref{e:1zq2}, we derive that
\begin{align}\label{e:dty1}
	\del{t} \hy = &   Y+(t+a) \del{t} Y \notag  \\
	= &   Y+(t+a) \Bigl[-3 z q Y- 2  q^{\frac{1}{2}} Y^{\frac{1}{2}} \Bigl(\frac{X^j}{X}\Phi_j\Bigr)\Bigr]\notag  \\
	= &   Y-3  Y+3  [1-(t+a)z q] Y  - 2 (t+a)  q^{\frac{1}{2}} Y^{\frac{1}{2}} \Bigl(\frac{X^j}{X}\Phi_j\Bigr)\notag  \\
	= &  -2 Y+3  \Bigl(1-\frac{A\hU}{\dq}\Bigr) Y -2(t+a)  q^{\frac{5}{2}} Y^{\frac{1}{2}} \Bigl(\frac{1}{q^2} \frac{X^j}{X}\Phi_j\Bigr)\notag  \\
< & 6 \sigma (t+a)^{- 1 }\hy + 2A^{-\frac{5}{2}}(t+a)^{-2 }      \dq^{\frac{5}{2}} \hy^{\frac{1}{2}}  \Bigl|\frac{1}{q^2}\frac{X^j}{X}\Phi_j\Bigr|.
\end{align}

By integrating \eqref{e:dty1}, using \eqref{e:intg1} and $A>(9G_1)^{\frac{1}{3}}$, we derive
\begin{align}\label{e:dty2}
	\hy(t,\xi) \leq & \hy_0 +6\sigma \int^t_0(s+a)^{- 1 }\hy ds + 2 A^{-\frac{5}{2}} \int^t_0 (s+a)^{-2 }  \dq^{\frac{5}{2}} \hy^{\frac{1}{2}} \Bigl|\frac{1}{q^2}\frac{X^j}{X}\Phi_j\Bigr| ds \notag  \\
	\leq & \frac{1}{2}  A\sigma^2    +6 A \sigma^3  \int^t_0(s+a)^{- 1 }  ds +  2 G_1 A^{-2}     \sigma \int^t_0 (s+a)^{-2 }    ds \notag  \\
	< & \frac{1}{2} A\sigma^{2}  +6 A   \Bigl(\frac{1}{10}+  \frac{9|H^\prime(0)|}{4 \beta E} \sigma\Bigr) \sigma^3    +  2 G_1 A^{-2}    \sigma^{2} \ <\frac{9}{10}A\sigma^{2}
\end{align}
for $t\in[0,T)$ by noting $\sigma\in(0,\sigma_\star)$.

\underline{Conclusions:}
Gather above estimates \eqref{e:dtq2}, \eqref{e:dtu2},  \eqref{e:dtv2} and \eqref{e:dty2}, we further conclude, from the bootstrap assumptions \eqref{e:btasp1}--\eqref{e:btasp3}, that
\begin{align}
	\dq  < & 1-\frac{1}{4} \sigma , \label{e:Concl1}\\
	\hU > & \frac{1-\frac{3}{2}\sigma}{ A} , \label{e:Concl2}\\
	\hV > &  -\frac{3}{28}     , \label{e:Concl3} \\
	\hy < & \frac{9}{10}A\sigma^{2} . \label{e:Concl4}
\end{align}
for $t\in[0,T)$ where $T$ is any constant such that $T\in(0,\min\{T^\natural,T^\prime\}]$, that is, all the bootstrap assumptions are improved. Hence, Proposition \ref{t:btpr}.$(1)$ is verified. Moreover,  Proposition \ref{t:btpr}.$(3)$ can be verified by noting the continuity of $\dq$, $\hV$, $\hU$ and $\hy$ (due to the continuity of $\rw^i$ and $\chi$, see \cite[Theorem $2.9$]{Liu2021d} and Proposition \ref{t:spprbd}, with the help of \eqref{e:nvar0a}--\eqref{e:nvar0d}, \eqref{e:dnvar1}--\eqref{e:dnvar3}). Therefore, by the bootstrap argument, Proposition \ref{t:btpr}, we know that \eqref{e:Concl1}--\eqref{e:Concl4} hold for  $t\in[0,\min\{T^\natural,T^\prime\})$.

At last, using \eqref{e:qup1},
\begin{equation*}\label{e:qup1b}
	q=\frac{zq^2}{zq}= \frac{(t+a)^{-1}\hU}{1-(t+a)^{- 1 } \hV}>  \frac{(1-2\sigma)(t+a)^{-1}}{ A(1+\frac{1}{7} (t+a)^{- 1 } )},
\end{equation*}
we can obtain the lower bound of $q$, i.e., the upper bound of $\chi$.

Further, with the help of \eqref{e:nvar1}--\eqref{e:nvar3}, it follows
\begin{gather*}
	A (t+a)<|\chi(t,\xi) |   < \frac{A}{1-2\sigma}(t+a)\Bigl[1+ \frac{1}{7} (t+a)^{-1}\Bigr], \\
	\frac{1-2\sigma}{ A}(t+a)^{-2}|\chi(t,\xi) |^2<z(t,\xi)  <(t+a)^{-1}\Bigl[1+\frac{1}{7}(t+a)^{-1}\Bigr]|\chi(t,\xi) |,  	\\
	X^2(t,\xi)  <  A \sigma^{ 2} (t+a)^{-1 }|\chi(t,\xi) |.
\end{gather*}
for any $(t,\xi)\in[0,\min\{T^\natural,T^\prime\})\times\del{}\Omega(0)$ where $T^\natural$ is given by Corollary \ref{e:supcri}.
Then we complete this proof.

\section{Strong blowup theorem of large, expanding and irregularly-shaped molecular clouds} \label{s:evdifpro}
After knowing the blowup phenomenon, we usually ask the positions and  types of singularities. This section, by constructing some ``nice'' boundary to remove singularities of the boundary, makes some efforts on this purpose. In order to do so, we first estimate the tidal force to be the order of $1/|\bx|^3$ if $\rrho \in  \mathrm{R}^{0}([0,T)\times\Omega(t),G_0)$ (see \S\ref{s:tdal}). Then by differentiating the ``limiting momentum conservation equation'' \eqref{e:NEul2} on the boundary with respect to $x^j$ (see \eqref{e:dwev} and \eqref{e:dwev2}), we derive  Newtonian version of Raychauduri's equations in terms of the expansion $\Theta$, shear $\Xi_{jk}$ and rotation $\Omega_{jk}$ (see \eqref{e:ray1}--\eqref{e:ray3}). Analyzing this system, we arive at $|\del{i}w^j|<\infty$ near the boundary and then conclude the strong blowup theorem.

\subsection{Estimates of the tidal force}\label{s:tdal}
Before the proof of Theorem \ref{t:mainthm}.$(4)$, we first give a lemma on the estimates of second derivatives of Newtonian potential $ \Phi$ which play an important role in the proof of Theorem \ref{t:mainthm}.$(4)$.

\begin{lemma}\label{t:dphix3}
	Suppose $ \Phi$ is the Newtonian potential  defined by \eqref{e:Newpot} and $T\in(0,\infty]$ is a constant,
	$\rrho \in C^1([0,T)\times \Rbb^3)\cap \mathrm{R}^{0}([0,T)\times\Omega(t), G_0)$,
	 $\Omega(t)$ is precompact and contains a neighbor of origin for every $t\in[0,T)$. Then the following estimate holds\footnote{Recall \S\ref{s:mtrx}, i.e., for two matrices $A$ and $B$, $A\leq B$ means $\xi^TA\xi\leq \xi^TB\xi$ for any vector $\xi$. }
	\begin{equation*}
		-  \frac{G_0}{|\bx|^{3 }}\delta_{ij} \leq  \del{i}\del{j} \Phi(t,\bx) 	\leq     \frac{G_0}{|\bx|^{3 }}\delta_{ij}
	\end{equation*}
	for any $(t, \bx) \in [0,T) \times  \del{}\Omega(t)$.
\end{lemma}
\begin{proof}
	The proof is similar to Lemma \ref{t:dphix2}. First, direct calculations, with the help of Lemma \ref{t:ddphi} leads to
	\begin{align}\label{e:ddphi5}
	\del{i}\del{j} \Phi(t,\bx)=  & \frac{1}{4\pi}\int_{\Omega(t)}   \rrho (t, \by)  \frac{\delta_{ij}}{|\bx -\by|^3} d^3\by -\frac{3}{4\pi} \int_{\Omega(t)}  \rrho (t, \by)  \frac{\delta_{kj}(x^k-y^k)\delta_{il}(x^l-y^l)}{|\bx -\by|^5} d^3\by
	\end{align}
	for all $(t, \bx) \in [0,T) \times \del{}\Omega(t)$.

	Since $\rrho\in \mathrm{R}^{0}([0,T)\times\Omega(t),G_0)$, by Definition \ref{t:diffdis},  implies there are constants  $\delta\in(0,1)$, such that $G_0:= \frac{1}{2\pi}\bigl(\frac{   M  }{\delta^3  } +3M \bigr)$ and
	for any  $t\in[0,T)$, every vector $\bx \in\del{}\Omega(t)$ and $\mathbf{r}\in B(\mathbf{n},\delta)$ where  $\mathbf{n}:=\bx/|\bx|$, we have
	\begin{align}\label{e:regdef1}
	\int_{B(\mathbf{n},\delta)}\frac{ |\bx|^{3} \rrho(t,|\bx|\br ) }{|\mathbf{n}-\mathbf{r}|^{3 }}d^3\mathbf{r}=\int_{B(\mathbf{n},\delta)}\frac{|\mathbf{n}-\mathbf{r}|^{-1}|\bx|^{3} \rrho(t,|\bx|\br ) }{|\mathbf{n}-\mathbf{r}|^2}d^3\mathbf{r}<3M.
	\end{align}	 	
	Then, let $\by=|\bx|\br$ and noting $\rrho (t, \bx)=\rrho (t, |\bx|\mathbf{n})=0$, we arrive at
	\begin{align}\label{e:mest}
	&\frac{1}{\pi} \int_{\Omega(t)\cap B(\bx,\delta|\bx|)}  \frac{\rrho (t, \by)}{|\bx -\by |^3}d^3\by
	=   \frac{1}{ \pi} \int_{B(\mathbf{n},\delta)\cap \{\br\;|\;|\bx|\br\in\Omega(t) \}}  \frac{\rrho (t, |\bx|\br)}{\bigl|\bx -|\bx|\br \bigr|^3}|\bx|^3 d^3\br  \notag  \\
	& \hspace{6cm} = \frac{1}{ \pi|\bx|^3} \int_{B(\mathbf{n},\delta) }  \frac{|\bx|^3\rrho (t, |\bx|\br)}{\bigl|\mathbf{n}-\mathbf{r}  \bigr|^3} d^3\br \overset{\eqref{e:regdef1}}{<} \frac{3M}{\pi |\bx|^3} .
	\end{align}
	
	For any $\xi^i\in \Rbb^3$, with the help of \eqref{e:ddphi5} and $0\leq (\xi^j\delta_{kj}(x^k-y^k)/|\bx-\by|)^2\leq| \xi |^2= \xi^i\xi^j\delta_{ij} $,
	\begin{align*}
		-\frac{1}{2\pi}\int_{\Omega(t)}   \rrho (t, \by)  \frac{\delta_{ij}\xi^i\xi^j}{|\bx -\by|^3} d^3\by \leq &\frac{1}{4\pi}\int_{\Omega(t)}   \rrho (t, \by)  \frac{\delta_{ij}\xi^i\xi^j}{|\bx -\by|^3} d^3\by \notag  \\
		& -\frac{3}{4\pi} \int_{\Omega(t)}  \rrho (t, \by)  \frac{\xi^j\delta_{kj}(x^k-y^k)\xi^i\delta_{il}(x^l-y^l)}{|\bx -\by|^5} d^3\by  \notag  \\
		\leq & \frac{1}{4\pi}\int_{\Omega(t)}   \rrho (t, \by)  \frac{\xi^i\xi^j\delta_{ij}}{|\bx -\by|^3} d^3\by.
	\end{align*}
	That is,
	\begin{align*}
		  -\frac{1}{2\pi}\int_{\Omega(t)}   \rrho (t, \by)  \frac{\delta_{ij}\xi^i\xi^j}{|\bx -\by|^3} d^3\by \leq \xi^i\xi^j\del{i}\del{j} \Phi(t,\bx) 	\leq & \frac{1}{4\pi}\int_{\Omega(t)}   \rrho (t, \by)  \frac{\xi^i\xi^j\delta_{ij}}{|\bx -\by|^3} d^3\by.
	\end{align*}

Then let us estimate $\frac{1}{\pi} \int_{\Omega(t)} \rrho(t,\by)  \frac{1}{|\bx-\by|^{3 }} d^3 \by$. From the definition,  $\by\in\Omega(t)\setminus B(\bx,\delta|\bx|)$ implies $|\bx-\by|\geq\delta|\bx|$, then for every $\bx\in \del{}\Omega(t)$, we have,
	\begin{align*}
	&\frac{1}{\pi} \int_{\Omega(t)}  \frac{ \rrho(t,\by) }{|\bx-\by|^{3 }} d^3 \by
    =  \frac{1}{ \pi} \int_{\Omega(t)\setminus B(\bx,\delta|\bx|)}   \frac{\rrho(t,\by) }{|\bx -\by |^3}d^3\by + \frac{1}{ \pi}  \int_{\Omega(t) \cap B(\bx,\delta|\bx|) }    \frac{ \rrho(t,\by) }{|\bx -\by |^{3}} d^3\by \notag \\
	&\hspace{2cm} \leq  \frac{1}{ \pi\delta^3}\frac{1}{|\bx |^3} \int_{\Omega(t)\setminus B(\bx,\delta|\bx|)}   \rrho(t,\by)  d^3 \by+ \frac{3M}{\pi|\bx|^3}
 	\leq   \Bigl(\frac{   M  }{\delta^3  } +  3M   \Bigr) \frac{1}{\pi|\bx|^{3 }} .
	\end{align*}
Then, we conclude
	\begin{align*}
		- \xi^i\xi^j\delta_{ij}\frac{G_0}{|\bx|^{3 }} \leq \xi^i\xi^j\del{i}\del{j} \Phi(t,\bx) 	\leq &  \xi^i\xi^j\delta_{ij}  \frac{G_0}{|\bx|^{3 }}.
	\end{align*}
This completes this proof.
	\end{proof}

\subsection{Strong blowup theorem}\label{s:sbt}
Before stating the strong blowup theorem, let us first prove a basic lemmas which have been mentioned in Remarks \ref{r:wlldefad2}. This lemma states any strong admissible initial boundary data is admissible initial boundary data.

\begin{lemma}\label{t:sadad}
	If $\xi\in\del{}\Omega(0)$,  $\sigma\in(0,\sigma_\star)$ are given by Definition \ref{t:paramtr}.\eqref{t:Def4}.
	Then 
	\begin{equation*}
		\mathfrak{A}^\dag(E,M,G_1,G_0,H^\prime(0),\sigma) \subset  \mathfrak{A}(E,M,G_1,H^\prime(0))
	\end{equation*}
\end{lemma}
\begin{proof}
	For any $(\mathbf{w}_0|_{ \del{}\Omega(0)}, \Omega(0))\in \mathfrak{A}^\dag(E,M,G_1,G_0,H^\prime(0),\sigma)$ where  $\sigma\in(0,\sigma_\star)$ and  $\xi\in\del{}\Omega(0)$ are given by Definition \ref{t:paramtr}. If we define $z_0=\lambda_1\lambda_0^{-1}\sigma  |\xi|$, then there is a constant $\sigma_\xi\equiv \sigma$, such that, by \eqref{e:BinO.a},
	\begin{align*}
		|\xi|>&  \lambda_0 (9G_1 )^{\frac{1}{3}} \sigma^{-1}=\lambda_0 (9G_1 )^{\frac{1}{3}} \sigma_\xi^{-1} > r_c.
	\end{align*}
The last inequality is because $\lambda_0(\sigma)\sigma^{-1}>\frac{2}{(2-\sigma)\sigma}>\frac{2}{(2-\sigma_\star)\sigma_\star}$.
	This implies Condition $(A)$ of Definition \ref{t:paramtr}.\eqref{t:Def3}, i.e., $\Omega(0)$ is a precompact set and $\Omega(0)\supsetneq B(0,r_c) $. 	
Then \eqref{e:BinO.a} and \eqref{e:xidn0.a} lead to \eqref{e:xidn0}. This is because
\begin{gather*}
	z_0=\lambda_1\lambda_0^{-1}\sigma  |\xi|\in\lambda_1\lambda_0^{-1}\sigma \times \biggl(\lambda_0 (9G_1 )^{\frac{1}{3}} \sigma^{-1}, \frac{\lambda_0}{24} \sqrt{\frac{\beta E}{M}}\sigma^{-1}\biggr)=\biggl(\lambda_1 (9G_1 )^{\frac{1}{3}} , \frac{\lambda_1}{24} \sqrt{\frac{\beta E}{M}} \biggr).
\end{gather*}
Then Condition $(B)$ of Definition \ref{t:paramtr}.\eqref{t:Def3} has been proven.

Condition $(C)$ of Definition \ref{t:paramtr}.\eqref{t:Def3} is apparently satisfied due to $(C^\star)$ of Definition \ref{t:paramtr}.\eqref{t:Def4}. After verifying all conditions $(A)$--$(C)$ of Definition \ref{t:paramtr}.\eqref{t:Def3}, we conclude $(\mathbf{w}_0|_{ \del{}\Omega(0)}, \Omega(0))\in \mathfrak{A}(E,M,G_1,H^\prime(0))$ and this completes the proof.
\end{proof}

Now let us state the strong blowup theorem which rules out certain  singularities on the boundary by assuming the strong admissible initial boundary data and the higher regularities of the solutions.
\begin{theorem}[Strong Blowup Theorem]  \label{t:blupthm}
Under assumptions of Local Existence Theorem (\cite[Theorem $2.12$]{Liu2021d}) with the initial data satisfies $(b)$ for $1<\gamma<2$ and $\Omega(0)$ is a precompact $C^1$-domain,
further suppose $(E,M,G_1)$ is any given compatible set,   $H^\prime(0)$, $G_0>0$ and $\sigma\in(0,\sigma_\dag)$, where $\sigma_\dag$ is given by \eqref{e:sigma0}, are given constants,  and  there is no the first class of global solution with $G_1>0$,
then
if $(\mathbf{w}_0|_{ \del{}\Omega(0)}, \Omega(0)) \in \mathfrak{A}^\dag(E,M,G_1,G_0,H^\prime(0),\sigma)$ is the strong admissible initial boundary data of the diffuse boundary $\del{}\Omega(0)$, then there is a constant $0<T^\star<+\infty$, such that,
the classical solution to the diffuse boundary problem breaks down at $t=T^\star$,
and if $T^\star\in(0,T^\natural]$ and $\rrho\in \mathrm{R}^0([0,T^\star)\times \Omega(t),G_0)$, then there is a small constant $\epsilon>0$, such that
	\begin{equation}\label{e:blup2}
		\int^{T^\star}_0 \Bigl( \|\Theta(s)\|_{L^\infty( \mathring{\Omega}_\epsilon(s) )}  +\|\Omega_{jk}(s)\|_{ L^\infty( \mathring{\Omega}_\epsilon(s) )}+\|\nabla \rrho^{\frac{\gamma-1}{2}}(s)\|_{ L^\infty(\Omega(s))} \Bigr) ds=+\infty.
	\end{equation}
\end{theorem}
Before this proof, we introduce a crucial lemma which is a key ingredient of this proof and postpone the proof of this lemma to the next section \S\ref{s:pflm}, analysis of Raychaudhuri's equations.
\begin{lemma}\label{t:anaRay}
Suppose $(\rrho,\rw^i,\Omega(t))$ is a classical solution for $t\in [0, T^\star) $ to the diffuse problem of Euler--Poisson equations \eqref{e:NEul1}--\eqref{e:ffbdry0} and $T^\star\in(0,T^\natural]$, $W_{jk}(t,\bx)$ is given by \eqref{e:WThOm0}, $(E,M,G_1)$ is any given compatible set,   $H^\prime(0)$, $G_0=2G_1>0$ and $\sigma\in(0,\sigma_\dag)$, where $\sigma_\dag$ is given by \eqref{e:sigma0}, are given constants, $\rrho\in \mathrm{R}^0([0,T^\star)\times \Omega(t),G_0)$ and this solution is not the first class of global solution with $G_1>0$,
then
if $(\mathbf{w}_0|_{ \del{}\Omega(0)}, \Omega(0)) \in \mathfrak{A}^\dag(E,M,G_1,G_0,H^\prime(0),\sigma)$ is the strong admissible initial boundary data of the diffuse boundary $\del{}\Omega(0)$, then
\begin{equation*}
	\|W_{jk} \|_{\Li([0,T^\star)\times\del{}\Omega(t))}<\infty.
\end{equation*}
\end{lemma}
\begin{remark}
 	The data given in terms of $z_0$ and $X_0$ are equivalent to the initial velocity $w^i(0,\bx)$ satisfying $\underline{w}^i(0,\xi):= \lambda_1\lambda_0^{-1}\sigma  \xi^i+X_0^i(\xi)$ (where $X_0^i(\xi):=X^i(0,\xi)$)
 for $\xi\in\del{}\Omega(0)$	where $\xi^i X_0^j\delta_{ij}=0$.
 Then direct calculations imply $X^2_0(\xi)=\delta_{jk} X_0^j(\xi) X_0^k(\xi)$ and  $z(0,\xi)=    \lambda_1\lambda_0^{-1}\sigma  | \xi |$.
 Then the derivative of $w^k$ is $\del{j} \underline{w}_i(0,\xi)= \lambda_1\lambda_0^{-1}\sigma \delta_{ji} + \del{j} X_{0i}$ ($X_{0i}:=X_0^j\delta_{ij}$),
 and in turn, the expansion, shear and rotation are
 \begin{gather*}
 	\Theta_0 =  3\lambda_1\lambda_0^{-1}\sigma  +   \del{j} X_0^j,  \quad \Theta_{0jk}= \lambda_1\lambda_0^{-1}\sigma \delta_{jk}+\frac{1}{2} (\del{j} X_{0k}+\del{k} X_{0j}),  \\
 	\Omega_{0jk}=  \frac{1}{2} (\del{k} X_{0j}-\del{j} X_{0k}) \AND
 	\Xi_{0jk}=   \frac{1}{2}(\del{j} X_{0k}+\del{k} X_{0j})-\frac{1}{3}(\del{i} X_0^i) \delta_{jk}.
\end{gather*}
\end{remark}
\begin{proof}
	By Local Existence and Uniqueness Theorems (see \cite[Theorem $2.12$ and Theorem $2.21$]{Liu2021d}), there is a constant $T>0$, such that the diffuse boundary problem $(i)$--$(iv)$ has a unique classical solution $(\rrho,\rw^i,\rPhi, \Omega(t)) $ defined by Definition \ref{t:clsl}. Moreover,  $\rrho^{\frac{\gamma-1}{2 }}
	\in C^{2}([0,T)\times\Rbb^3)$ and $ \Phi \in  C^2([0,T)\times\mathbb{R}^3)$, and the solution $\rw^i$ satisfies
	\begin{equation*}
		\rw^i\in C^0([0,T),H^{4}(\Omega(t)))\cap C^1([0,T),H^{3}(\Omega(t))).
	\end{equation*}

Since there is no the first class of global solution with
parameter $G_1$, and by Lemma \ref{t:sadad}, the initial data
\begin{equation*}
	(\mathbf{w}_0|_{ \del{}\Omega(0)}, \Omega(0)) \in  \mathfrak{A}^\dag(E,M,G_1,G_0,H^\prime(0),\sigma) \subset  \mathfrak{A}(E,M,G_1,H^\prime(0)),
\end{equation*}
for any $\sigma\in(0,\sigma_\dag)$ and any $\xi\in\del{}\Omega(0)$, that is, the conditions of Weak Blowup Theorem \ref{t:blupthm0} are satisfied, and it follows  this solution is not the global solution to the diffuse boundary problem, and there is a constant $0<T^\star<+\infty$, such that, the classical solution breaks down at $t=T^\star$ and
\begin{align}\label{e:bl1}
	\int^{T^\star}_0 \Bigl(\|\nabla w^i(s)\|_{L^\infty( \Omega(s))}  +\|\nabla \rrho^{\frac{\gamma-1}{2}}\|_{ L^\infty(\Omega(s))} \Bigr) ds=+\infty.
\end{align}

In the following, we focus on the proof of the improved estimate of blowups \eqref{e:blup2} if $T^\star\in(0,T^\natural]$ and $\rrho\in \mathrm{R}^0([0,T^\star)\times \Omega(t),G_0)$ (improve \eqref{e:bl1}). We prove it by contradictions and assume, otherwise,
\begin{equation}\label{e:oppoassp}
	\int^{T^\star}_0 \Bigl( \|\Theta(s)\|_{L^\infty( \mathring{\Omega}_\epsilon(s) )}  +\|\Omega_{jk}(s)\|_{ L^\infty( \mathring{\Omega}_\epsilon(s) )}+\|\nabla \rrho^{\frac{\gamma-1}{2}}(s)\|_{ L^\infty(\Omega(s))} \Bigr) ds<+\infty.
\end{equation}
By Lemma \ref{t:anaRay}, we arrive at $|W_{jk}(t,\bx)|<+\infty$ for any $(t,\bx)\in [0,T^\star)\times\del{}\Omega(t)$. Since $(\rrho,\rw^i,\Omega(t))$ is a classical solution for $t\in [0, T^\star) $, by the definition of classical solution (see Definition \ref{t:clsl}),
\begin{equation*}
W_{jk}(\check{x}^\mu)=w_{k,j}(\check{x}^\mu )=\lim_{x^\mu\rightarrow\check{x}^\mu}\del{j} \rw_k(x^\mu)=\lim_{x^\mu\rightarrow\check{x}^\mu}W_{jk}(x^\mu)<\infty	
\end{equation*}
for every $\check{x}^\mu\in[0,T^\star)\times \del{}\Omega(t)$ and $x^\mu\in[0,T^\star)\times \Omega(t)$. Then there is a small constant $\epsilon_0>0$, such that for any $\epsilon\in(0,\epsilon_0)$,  $\|W_{jk}(t)\|_{L^\infty( \Omega_\epsilon(t))}<C<\infty$ for $t\in[0,T^\star)$. Furthermore, integrating it for time, we obtain 
\begin{equation}\label{e:wbd}
	\int^{T^\star}_0 \|W_{jk}(s)\|_{L^\infty( \Omega_\epsilon(s))} ds<C T^\star <+\infty.
\end{equation}
On the other hand, by the strong continuation principle (see \cite[Theorem $3.3$]{Liu2021d}), with the help of the assumption \eqref{e:oppoassp} and \eqref{e:wbd}, we conclude there is a constant $\epsilon>0$ such that a priori estimate holds
\begin{equation*}
\int^{T^\star}_0 \Bigl(\|W_{jk}(s)\|_{L^\infty( \Omega_\epsilon(s))}+\|\Theta(s)\|_{L^\infty( \mathring{\Omega}_\epsilon(s) )} +\|\Omega_{jk}(s)\|_{ L^\infty( \mathring{\Omega}_\epsilon(s) )}+\|\nabla \alpha(s)\|_{ L^\infty(\Omega(s))} \Bigr) ds<\infty.
\end{equation*}
Then there is a constant $T^\prime>T^\star>0$, such that the classical solution $(\rrho,\rw^i,\rPhi,\Omega(t))$ exists on $t\in[0,T^\prime)$. This contradicts with the fact that $[0, T^\star)$ is the maximal time interval of the existence of the solution. Therefore, this contradiction implies there is a small constant $\epsilon>0$, such that
\begin{equation*}
	\int^{T^\star}_0 \Bigl( \|\Theta(s)\|_{L^\infty( \mathring{\Omega}_\epsilon(s) )}  +\|\Omega_{jk}(s)\|_{ L^\infty( \mathring{\Omega}_\epsilon(s) )}+\|\nabla \rrho^{\frac{\gamma-1}{2}}(s)\|_{ L^\infty(\Omega(s))} \Bigr) ds=+\infty.
\end{equation*}
This completes the proof.
\end{proof}

\subsection{Analysis of Newtonian version of Raychaudhuri's equations}\label{s:raych}
The main objective of this section is to prove Lemma \ref{t:anaRay} (i.e., the analysis of Newtonian version of Raychaudhuri's equations). We first derive Newtonian version of Raychaudhuri's equations (see \S\ref{s:51}), then point out an approximation solution to the  Raychaudhuri's equations (see \S\ref{s:52}). In the end, we prove Lemma \ref{t:anaRay} by analyzing the Newtonian version of Raychaudhuri's equations (see \S\ref{s:pflm}).

\subsubsection{Newtonian version of Raychaudhuri's equations}\label{s:51}
Since $\rrho^{\gamma-1}\in C^2([0,T)\times\Rbb^3)$,  $\rw^i\in C^0([0,T),H^{4}(\Omega(t)))\cap C^1([0,T),H^{3}(\Omega(t)))$ and $w^i(t,\cdot)\in H^s(\Omega(t))\subset C^{s-2}(\overline{\Omega(t)})$ ($s\geq 3$) (e.g., see \cite[Theorem $4.12$]{Adams2003b}, Sobolev Imbedding Theorem), by multiplying $1/\rrho$ on both sides of \eqref{e:NEul2}, and differentiating it with respect to $x^j$, we derive that, for any $(t,\bx)\in[0,T)\times \Omega(t)$,
\begin{align}\label{e:dwev}
	\del{t}\del{j}w^k  +w^l\del{l}\del{j} w^k+\del{j}w^l\del{l} w^k +\del{j}\Bigl(\frac{1}{\rho}\del{}^k p\Bigr)=  -  \del{}^k\del{j} \Phi
\end{align}
where we calculate
\begin{equation*}
	\del{j}\Bigl(\frac{1}{\rho}\del{}^k p\Bigr)=\delta^{ki} \del{j} \del{i} \rrho^{\gamma-1}.
\end{equation*}
Since $\rrho^{\gamma-1}\in C^2([0,T)\times\Rbb^3)$ and $\del{j} \del{i} \rrho^{\gamma-1}\equiv 0$ for $(t, \bx)\in \overline{\Omega}^\mathsf{c}(t)$,
\begin{equation*}
	\lim_{(t,\bx)\rightarrow (\check{t},\check{\bx})} \del{j}\Bigl(\frac{1}{\rho}\del{}^k p\Bigr)=\lim_{(t,\bx)\rightarrow (\check{t},\check{\bx})}\delta^{ki} \del{j} \del{i} \rrho^{\gamma-1}=0
\end{equation*}
for any $(\check{t},\check{\bx})\in [0,T)\times\del{}\Omega(t)$
and $(t, \bx)\in [0,T)\times\Omega(t)$. Therefore, by noting, see \cite[Lemma $A.2$]{Liu2021d}, $\lim_{(t,\bx)\rightarrow (\check{t},\check{\bx})}D_t=D_t\lim_{(t,\bx)\rightarrow (\check{t},\check{\bx})}$, and using $\lim_{(t,\bx)\rightarrow (\check{t},\check{\bx})}$ acting on the both sides of \eqref{e:dwev}, we obtain
\begin{align}\label{e:dwev2}
	D_tw^k_{,j} (\check{t},\check{\bx})  + w^l_{,j}(\check{t},\check{\bx})  w^k_{,l}(\check{t},\check{\bx})  =  -  \del{}^k\del{j} \Phi(\check{t},\check{\bx})
\end{align}
for any $(\check{t},\check{\bx})\in [0,T)\times\del{}\Omega(t)$.
This implies that
\begin{align*}
\underline{D_t W^i_k}(t,\xi)=\del{t} \uW^i_k(t,\xi)=-\uW^i_j(t,\xi) \uW_k^j(t,\xi)- \underline{\delta^{il} \del{k} \del{l} \Phi }(t,\xi),
\end{align*}
that is, by using $\delta_{ij}$ to  lower the indecies
\begin{align}\label{e:dtweq1}
\underline{D_t W_{jk}}(t,\xi)=\del{t}\uW_{jk}(t,\xi)=-\delta^{li}\uW_{ji}(t,\xi)\uW_{lk}(t,\xi) - \underline{\del{k}\del{j} \Phi}(t,\xi).
\end{align}
for any $(t,\xi)\in [0,T)\times \del{}\Omega(0)$.
Using material derivative $D_t$ acting on \eqref{e:WThOm0a} and using  \eqref{e:dtweq1} to substitute $D_tW_{jk}$, with the help of \eqref{e:WThOm1}, we derive, for any $(\check{t},\check{\bx})\in [0,T)\times\del{}\Omega(t)$,
\begin{align}\label{e:deeq0}
	D_t\Theta_{jk}= & \frac{1}{2}\bigl(D_t W_{jk} + D_t W_{kj} \bigr) 
	=   -\frac{1}{2}(\delta^{li}W_{ji}W_{lk}+\delta^{li}W_{kl}W_{ij})-\del{k}\del{j} \Phi \notag  \\
	= & -\frac{1}{2}[\delta^{li}(\Theta_{ji}-\Omega_{ji})(\Theta_{lk}-\Omega_{lk})+\delta^{li}(\Theta_{kl}-\Omega_{kl})(\Theta_{ij}-\Omega_{ij})]- \del{k}\del{j} \Phi\notag  \\
	= & - \delta^{li}( \Theta_{ji}\Theta_{lk}+ \Omega_{ji}\Omega_{lk} ) - \del{k}\del{j} \Phi.
\end{align}
Similarly, by \eqref{e:WThOm0b}, \eqref{e:WThOm1} and \eqref{e:dtweq1}, we calculate $D_t\Omega_{jk} $, for any $(\check{t},\check{\bx})\in [0,T)\times\del{}\Omega(t)$,
\begin{align}\label{e:roeq0}
	&D_t\Omega_{jk} =  \frac{1}{2}\bigl(D_t W_{kj} -D_t W_{jk} \bigr)  
	=   -\frac{1}{2}\delta^{li}(W_{ki}W_{lj}-W_{jl}W_{ik}) \notag  \\
	&\hspace{0.5cm}=  -\frac{1}{2}\delta^{li}[(\Theta_{ki}-\Omega_{ki})(\Theta_{lj}-\Omega_{lj})-(\Theta_{jl}-\Omega_{jl})(\Theta_{ik}-\Omega_{ik})] 
	=   - \delta^{li} ( \Theta_{jl}\Omega_{ik}+\Omega_{jl}\Theta_{ik}).
\end{align}

Further, we decompose $\Theta_{jk}$ to be the trace part $\Theta$ (called the \textit{expansion}, also see \eqref{e:theta}) and the traceless component $\Xi_{jk}$ (called the \textit{shear}),
\begin{align}\label{e:defxi}
	\Theta:=\Theta_{jk}\delta^{jk} \AND \Xi_{jk}:=\Theta_{jk}-\frac{1}{3}\Theta \delta_{jk}.
\end{align}
Then
\begin{align}\label{e:defth}
	\Theta_{jk}=\Xi_{jk}+\frac{1}{3}\Theta \delta_{jk}.
\end{align}
Then, by \eqref{e:WThOm1} and \eqref{e:defth}, $W_{kj}$ can be decomposed into
\begin{equation}\label{e:Whelmd}
	W_{kj}=\Xi_{jk}+\frac{1}{3}\Theta \delta_{jk}+\Omega_{jk} \AND W_{jk}=\Xi_{jk}+\frac{1}{3}\Theta \delta_{jk}-\Omega_{jk}.
\end{equation}

We first note that $\rrho(\check{t},\check{\bx})\equiv 0$ for any $(\check{t},\check{\bx})\in [0,T)\times\del{}\Omega(t)$, and  multiplying \eqref{e:deeq0} by $\delta^{jk}$ and noting that $\delta^{lk}\Xi_{lk}=0$ (since $\Xi_{lk}$ is traceless symmetric matrix) yield the equation of the expansion $\Theta$, with the help of \eqref{e:defth} and \eqref{e:NEul3},
\begin{align}\label{e:rayd1}
	D_t\Theta = & - \delta^{jk}\delta^{li}( \Theta_{ji}\Theta_{lk}+ \Omega_{ji}\Omega_{lk} ) - \delta^{jk}\del{k}\del{j} \Phi \notag  \\
	= & -\delta^{jk}\delta^{li}\Bigl[ \Bigl(\Xi_{ji}+\frac{1}{3}\Theta \delta_{ji}\Bigr)\Bigl(\Xi_{lk}+\frac{1}{3}\Theta \delta_{lk}\Bigr)+ \Omega_{ji}\Omega_{lk} \Bigr] - \Delta \Phi \notag  \\
	= & -\Bigl[ \delta^{jk}\delta^{li}\Xi_{ji}\Xi_{lk}+\frac{1}{3}\Theta  \delta^{lk}\Xi_{lk}+\frac{1}{3}\Theta  \delta^{ji} \Xi_{ji} +\frac{1}{9}\Theta^2 \delta_{lk}   \delta_{ji}\delta^{jk}\delta^{li}  + \delta^{jk}\delta^{li}\Omega_{ji}\Omega_{lk} \Bigr] -  \rrho  \notag  \\
	= & -\Bigl[ \delta^{jk}\delta^{li}\Xi_{ji}\Xi_{lk}  +\frac{1}{3}\Theta^2   + \delta^{jk}\delta^{li}\Omega_{ji}\Omega_{lk} \Bigr] -  \rrho  
	=   -\frac{1}{3}\Theta^2 -  \delta^{jk}\delta^{li}\Xi_{ji}\Xi_{lk}     - \delta^{jk}\delta^{li}\Omega_{ji}\Omega_{lk}
\end{align}
for any $(\check{t},\check{\bx})\in [0,T)\times\del{}\Omega(t)$.
The following equation is obtained by differentiating \eqref{e:defxi} and substituting $\Theta_{jk}$ and $\Theta$ by \eqref{e:deeq0} and \eqref{e:rayd1}, respectively,
\begin{align}\label{e:rayd2}
	D_t\Xi_{jk} =   D_t\Theta_{jk}-\frac{1}{3}\delta_{jk} D_t \Theta 
	= & - \delta^{li}( \Theta_{ji}\Theta_{lk}+ \Omega_{ji}\Omega_{lk} ) - \del{k}\del{j} \Phi-\frac{1}{3}\delta_{jk}\Bigl(-\frac{1}{3}\Theta^2 \notag  \\
	& -  \delta^{mn}\delta^{li}\Xi_{mi}\Xi_{ln}     - \delta^{mn}\delta^{li}\Omega_{mi}\Omega_{ln} \Bigr) \notag  \\
	= & - \delta^{li}\Bigl[ \Bigl(\Xi_{ji}+\frac{1}{3}\Theta \delta_{ji}\Bigr)\Bigl(\Xi_{lk}+\frac{1}{3}\Theta \delta_{lk}\Bigl)+ \Omega_{ji}\Omega_{lk} \Bigr] \notag  \\
	& - \del{k}\del{j} \Phi-\frac{1}{3}\delta_{jk}\Bigl(-\frac{1}{3}\Theta^2 -  \delta^{rs}\delta^{li}\Xi_{ri}\Xi_{ls}     - \delta^{rs}\delta^{li}\Omega_{ri}\Omega_{ls} \Bigr)\notag  \\ 	
	&\hspace{-5cm} =  -\frac{2}{3}\Theta \Xi_{jk}  -   \delta^{li}\Xi_{ji}\Xi_{lk} - \delta^{li}\Omega_{ji}\Omega_{lk} +\frac{1}{3}\delta_{jk}  \delta^{rs}\delta^{li}\Xi_{ri}\Xi_{ls}  +\frac{1}{3}\delta_{jk} \delta^{rs}\delta^{li}\Omega_{ri}\Omega_{ls}   - \del{k}\del{j} \Phi
\end{align}
for any $(\check{t},\check{\bx})\in [0,T)\times\del{}\Omega(t)$.
Rewriting \eqref{e:roeq0} in terms of $\Theta$, $\Omega_{jk}$ and $\Xi_{jl}$, we obtain
\begin{align}\label{e:roeq1}
	D_t\Omega_{jk}
	= & - \delta^{li}( \Theta_{jl}\Omega_{ik}+\Omega_{jl}\Theta_{ik}) 
	=   - \delta^{li}[ (\Xi_{jl}+\frac{1}{3}\Theta \delta_{jl})\Omega_{ik}+\Omega_{jl}(\Xi_{ik}+\frac{1}{3}\Theta \delta_{ik})] \notag  \\
	= & -\frac{2}{3}\Theta \Omega_{jk} -    \delta^{li}\Xi_{jl}\Omega_{ik} - \delta^{li}\Omega_{jl}\Xi_{ik}.
\end{align}
for any $(\check{t},\check{\bx})\in [0,T)\times\del{}\Omega(t)$.
Then, gathering \eqref{e:rayd1}, \eqref{e:rayd2} and \eqref{e:roeq1}, adopting the notation given in \S\ref{s:lag}, it follows a \textit{Newtonian version of Raychaudhuri equations} (see, for example,  \cite{Wald2010} for Raychaudhuri equations in \textit{general relativity}),
\begin{align}
	\del{t}\uTheta= & -\frac{1}{3}\uTheta^2 -  \delta^{jk}\delta^{li}\uXi_{ji}\uXi_{lk}     - \delta^{jk}\delta^{li}\uOmega_{ji}\uOmega_{lk}, \label{e:ray1} \\
	\del{t}\uXi_{jk} = & -\frac{2}{3}\uTheta \uXi_{jk}  -   \delta^{li}\uXi_{ji}\uXi_{lk} - \delta^{li}\uOmega_{ji}\uOmega_{lk} +\frac{1}{3}\delta_{jk}  \delta^{rs}\delta^{li}\uXi_{ri}\uXi_{ls} +\frac{1}{3}\delta_{jk} \delta^{rs}\delta^{li}\uOmega_{ri}\uOmega_{ls}   - \underline{\del{k}\del{j} \Phi}, \label{e:ray2} \\
	\del{t}\uOmega_{jk} = & -\frac{2}{3}\uTheta \uOmega_{jk} -    \delta^{li}\uXi_{jl}\uOmega_{ik} - \delta^{li}\uOmega_{jl}\uXi_{ik}.  \label{e:ray3}
\end{align}

\subsubsection{Approximation solutions}\label{s:52}
Direct calculations imply that
\begin{equation}
(\tilde{\Theta}, \tilde{\Xi}_{jk}, \tilde{\Omega}_{jk}):=\biggl( \Bigl(\frac{\lambda_0}{3\lambda_1}\sigma^{-1}+\frac{1}{3} t \Bigr)^{-1}, 0, 0\biggr) \label{e:iWrs1a}
\end{equation}
solves the following ordinary differential system which is a similar system to \eqref{e:ray1}--\eqref{e:ray3} but without gravity (i.e., there is no tidal force in the following equation of $\tilde{\Xi}_{jk}$ comparing with \eqref{e:ray2}),
\begin{align*}
\del{t}\tilde{\Theta}= & -\frac{1}{3}\tilde{\Theta}^2 -  \delta^{jk}\delta^{li}\tilde{\Xi}_{ji} \tilde{\Xi}_{lk}     - \delta^{jk}\delta^{li}\tilde{\Omega}_{ji}\tilde{\Omega}_{lk}, 
\\
\del{t}\tilde{\Xi}_{jk} = & -\frac{2}{3}\tilde{\Theta} \tilde{\Xi}_{jk}  -   \delta^{li}\tilde{\Xi}_{ji}\tilde{\Xi}_{lk} - \delta^{li}\tilde{\Omega}_{ji}\tilde{\Omega}_{lk} +\frac{1}{3}\delta_{jk}  \delta^{rs}\delta^{li}\tilde{\Xi}_{ri}\tilde{\Xi}_{ls}  +\frac{1}{3}\delta_{jk} \delta^{rs}\delta^{li}\tilde{\Omega}_{ri}\tilde{\Omega}_{ls} , 
 \\
\del{t}\tilde{\Omega}_{jk} = & -\frac{2}{3}\tilde{\Theta} \tilde{\Omega}_{jk} -    \delta^{li}\tilde{\Xi}_{jl}\tilde{\Omega}_{ik} - \delta^{li}\tilde{\Omega}_{jl}\tilde{\Xi}_{ik}.  
\end{align*}

In the following proof of Lemma \ref{t:anaRay}, we will find solutions of Rauchaudhuri's equations in a small  neighborhood of this solution \eqref{e:iWrs1a}.

\subsubsection{Proof of Lemma \ref{t:anaRay}}\label{s:pflm}
The idea of this proof is first to  introduce a set of new variables which characterize the deviations of the solutions of the Raychaudhuri's equation from above approximation solution $(\tilde{\Theta}, \tilde{\Xi}_{jk}, \tilde{\Omega}_{jk})$. Then using bootstrap argument concludes these variables remain small if the initial data is suitably small. Then by transforming the original variables back, they are bounded.

\underline{Step $1$. Variable transformations:} Let us first introduce new variables,
\begin{align}
\mathfrak{e}(t,\xi):=&\uTheta^{-1}(t,\xi)-\frac{1}{3} t -\frac{\lambda_0}{3\lambda_1}\sigma^{-1} , \label{e:Wrs1}\\
\ks_{jk}(t,\xi):=& \uTheta^{-\frac{7}{4}}(t,\xi)\uXi_{jk}(t,\xi),  \label{e:Wrs2b} \\
\kb_{jk}(t,\xi):=&\uTheta^{-2}(t,\xi)\uOmega_{jk}(t,\xi)
, \label{e:Wrs3}
\end{align}
and denote
\begin{align}\label{e:Wrs4}
\kss:=\ks^{jk}\ks_{jk}= \ks_{11}^2+ \ks_{22}^2+ \ks_{33}^2+2 \ks_{12}^2+2 \ks_{13}^2+2 \ks_{23}^2.
\end{align}
Then this directly yields
\begin{align}\label{e:Wrs5}
\ks^2_{jk} \leq \kss \AND |\ks_{jk}| \leq \kss^{\frac{1}{2}}
\end{align}
for every $k,j$.
For later convenience, we express $\Theta$, $\Xi_{jk}$ and $\Omega_{jk}$ by \eqref{e:Wrs1}--\eqref{e:Wrs3},
\begin{align}
\uTheta(t,\xi)=& \Bigl(\ke(t,\xi)+\frac{\lambda_0}{3\lambda_1}\sigma^{-1} +\frac{1}{3} t \Bigr)^{-1} , \label{e:iWrs1}\\
\uXi_{jk}(t,\xi)=&   \uTheta^{\frac{7}{4}}(t,\xi) \ks_{jk} (t,\xi) ,\label{e:iWrs2b}  \\
\uOmega_{jk}(t,\xi)=&
\uTheta^{2}(t,\xi) \kb_{jk}(t,\xi). \label{e:iWrs3}
\end{align}
The identities $\delta^{lk}\Xi_{lk} \equiv 0$ (since $\Xi_{lk}$ is traceless symmetric matrix) and $\Omega_{ii} \equiv 0$ (since $\Omega_{jk}$ is antisymmetric, i.e., $\Omega_{jk}=-\Omega_{kj}$) lead to
\begin{align}\label{e:s0}
\delta^{jk} \ks_{jk}=0 \AND \delta^{jk} \kb_{jk}= \uTheta^{-2}\delta^{jk}\uOmega_{jk}=0.
\end{align}

By the assumptions of Theorem \ref{t:blupthm}, $(\mathbf{w}_0|_{ \del{}\Omega(0)}, \Omega(0)) \in \mathfrak{A}^\dag(E,M,G_1,G_0,H^\prime(0),\sigma)$. In terms of these new variables, this implies the \underline{initial data} satisfy (note the following data \eqref{e:dtdw} are equivalent to \eqref{e:deldata.a} of the strong admissible initial boundary data )
\begin{align}\label{e:dtdw}
|\ke_0|\leq \frac{\lambda_0}{12\lambda_1}\sigma^{-1} , \quad  \kss_{0} \leq \frac{1}{16} \sigma^{-1} \AND |\kb_{0jk}| \leq \frac{1}{4} \sigma^{-\frac{1}{2}},
\end{align}
then we \textit{claim} it follows estimates
\begin{align}\label{e:dtdw2}
|\ke|\leq \frac{\lambda_0}{6\lambda_1}\sigma^{-1}, \quad  \kss  \leq  \sigma^{-1} \AND |\kb_{jk}| \leq  \sigma^{-\frac{1}{2 }}.
\end{align}
for any $(t,\xi)\in[0,T^\star)\times \del{}\Omega(0)$.  \eqref{e:dtdw2} further implies the conclusion of this Lemma \ref{t:anaRay}, i.e.,  $|W_{jk}(t,\bx)|<\infty$ for any $(t,\bx)\in [0,T^\star)\times\del{}\Omega(t)$.
Since by \eqref{e:Whelmd}, \eqref{e:iWrs1}--\eqref{e:iWrs3}, we obtain
\begin{equation}\label{e:Wkj}
\underline{W}_{kj}= \uTheta^{\frac{7}{4}} \ks_{jk}+\frac{1}{3}\uTheta \delta_{jk}+\uTheta^{2} \kb_{jk}
\end{equation}
and we are able to estimate $|\uTheta|$, with the help of \eqref{e:dtdw2}, by
\begin{equation}\label{e:thest0}
|\uTheta|\leq \Bigl|\Bigl(\ke+\frac{\lambda_0}{3\lambda_1}\sigma^{-1} \Bigr)^{-1} \Bigr|\leq  \Bigl( \frac{\lambda_0}{3\lambda_1}\sigma^{-1} - \frac{\lambda_0}{6\lambda_1}\sigma^{-1} \Bigr)^{-1} = \frac{6\lambda_1}{\lambda_0}
\sigma .
\end{equation}
Then by \eqref{e:dtdw2}--\eqref{e:thest0}, we derive an estimate
\begin{equation*}
|\underline{W}_{kj}|\leq \Bigl(\frac{6\lambda_1}{\lambda_0}\Bigr)^{\frac{7}{4 }}   \sigma^{\frac{5}{4 }} + \frac{2\lambda_1}{\lambda_0} \sigma  + \Bigl(\frac{6\lambda_1}{\lambda_0}\Bigr)^2 \sigma^{\frac{3}{2 } } <+\infty
\end{equation*}
for every $(t,\xi)\in[0,T^\star)\times \del{}\Omega(0)$, i.e.,  $\|W_{jk} \|_{\Li([0,T^\star)\times\del{}\Omega(t))}<\infty$. Therefore, in the rest of the proof, we only need to verify above claim,  i.e., \eqref{e:dtdw2} for $(t,\xi)\in[0,T^\star)\times \del{}\Omega(0)$.

\underline{Step $2$. The bootstrap argument:}  Now let us focus on the claim, i.e., \eqref{e:dtdw2} for $(t,\xi)\in[0,T^\star)\times \del{}\Omega(0)$, and we use bootstrap argument to prove it. In order to do this, we give the \underline{bootstrap assumptions} as \eqref{e:dtdw2}:
\begin{align*}
|\ke|\leq \frac{\lambda_0}{6\lambda_1}\sigma^{-1}, \quad  \kss  \leq  \sigma^{-1} \AND |\kb_{jk}| \leq  \sigma^{-\frac{1}{2 }}
\end{align*}
for $(t,\xi)\in[0,T)\times \del{}\Omega(0)$ where $T$ is any constant such that $T\in(0, T^\star]$.

Let us first calculate a simple integration and estimate it for later use, by using \eqref{e:dtdw2} and assuming $ b >0$ is any constant,
\begin{align}\label{e:speint}
&\int_0^t \uTheta^{1+ b } (s,\xi)  ds = \int_0^t\Bigl(\ke+\frac{\lambda_0}{3\lambda_1}\sigma^{-1}+\frac{1}{3} s \Bigr)^{-1- b }   ds
 \notag \\
& \hspace{2cm} \leq \int^\infty_0\Bigl(\frac{\lambda_0}{6\lambda_1}\sigma^{-1} +\frac{1}{3}  s \Bigr)^{-1- b }ds
=  \frac{3}{  b  }\Bigl(\frac{6\lambda_1}{\lambda_0}\Bigr)^{ b } \sigma^{b}.
\end{align}

Now let us differentiate \eqref{e:Wrs1}, with the help of  \eqref{e:ray1} and  \eqref{e:iWrs2b}--\eqref{e:iWrs3}, we arrive at
\begin{align}\label{e:Wrseq1}
\del{t}\ke=&-\uTheta^{-2}\del{t}\uTheta-\frac{1}{3}  
= -\uTheta^{-2}\Bigl(-\frac{1}{3}\uTheta^2 -  \delta^{jk}\delta^{li}\uXi_{ji}\uXi_{lk}     - \delta^{jk}\delta^{li}\uOmega_{ji}\uOmega_{lk}\Bigr)-\frac{1}{3} \notag  \\
=& \uTheta^{ \frac{3}{2}} \delta^{jk}\delta^{li} \ks_{lk}  \ks_{ji}  +  \delta^{jk}\delta^{li} \uTheta^{ 2 }  \kb_{ji}\kb_{lk}  
=  \uTheta^{\frac{3}{2}} \kss  +\delta^{jk}\delta^{li} \uTheta^{ 2 }  \kb_{ji}\kb_{lk}
\end{align}
for $(t,\xi)\in[0,T)\times \del{}\Omega(0)$.
Then integrating \eqref{e:Wrseq1} implies
\begin{align}\label{e:eeq}
\ke = & \ke_0+\int_0^t \uTheta^{\frac{3}{2}} \kss ds    + \int_0^t \uTheta^{ 2}  \delta^{jk}\delta^{li}  \kb_{ji} \kb_{lk} ds.
\end{align}
for $(t,\xi)\in[0,T)\times \del{}\Omega(0)$.

Then estimating $|\ke|$ by \eqref{e:eeq}, with the help of \eqref{e:speint}, $0<\sigma<\sigma_\dag$ (recall $\sigma_\dag$ is defined by \eqref{e:sigma0}) and noting (since $0<\sigma<1/5$) \begin{equation}\label{e:Rest1}
	\frac{6\lambda_1}{\lambda_0} \in\Bigl(4,\frac{87}{14}\Bigr),
\end{equation}
yields
\begin{align*}
|\ke| \leq & |\ke_0|+\int_0^t \uTheta^{\frac{3}{2}} |\kss|  ds    + \int_0^t \uTheta^{ 2}  |\delta^{jk}\delta^{li}  \kb_{ji} \kb_{lk}| ds  
\leq   \frac{\lambda_0}{12\lambda_1}\sigma^{-1} +  \sigma^{-1} \int_0^t\uTheta^{\frac{3}{2}}  ds   + 9  \sigma^{-1}  \int_0^t\uTheta^{ 2}   ds   \notag  \\
\leq & \frac{\lambda_0}{12\lambda_1}\sigma^{-1} +6   \Bigl(\frac{6\lambda_1}{\lambda_0}\Bigr)^{\frac{1}{2}} \sigma^{\frac{1}{2 } } \sigma^{-1}  + 27  \Bigl(\frac{6\lambda_1}{\lambda_0}\Bigr)\sigma \sigma^{-1}
< \frac{\lambda_0}{8\lambda_1}\sigma^{-1}
\end{align*}
for $(t,\xi)\in[0,T)\times \del{}\Omega(0)$.

In order to estimate $\kss $, we first differentiate \eqref{e:Wrs2b} by using \eqref{e:ray1}--\eqref{e:ray2}, \eqref{e:iWrs2b}--\eqref{e:iWrs3},  and it yields
\begin{align}
\del{t}\ks_{jk}
= &  -\frac{7}{4}\uTheta^{-2-\frac{3}{4}} \uXi_{jk} \del{t}\uTheta  +\uTheta^{-1-\frac{3}{4}}\del{t} \uXi_{jk}   \notag \\
= &  -\frac{7}{4}\uTheta^{-2-\frac{3}{4}} \uXi_{jk} \Bigl(-\frac{1}{3}\uTheta^2 -  \delta^{jk}\delta^{li}\uXi_{ji}\uXi_{lk}     - \delta^{jk}\delta^{li}\uOmega_{ji}\uOmega_{lk}\Bigr)  +\uTheta^{-1-\frac{3}{4}}\Bigl(-\frac{2}{3}\uTheta \uXi_{jk}   \notag  \\
& - \delta^{li}\uXi_{ji}\uXi_{lk} - \delta^{li}\uOmega_{ji}\uOmega_{lk} +\frac{1}{3}\delta_{jk}  \delta^{rs}\delta^{li}\uXi_{ri}\uXi_{ls} +\frac{1}{3}\delta_{jk} \delta^{rs}\delta^{li}\uOmega_{ri}\uOmega_{ls}   - \underline{\del{k}\del{j} \Phi }\Bigr) \notag  \\
=
& -\frac{1}{12}  \uTheta \ks_{jk}  + \biggl[  \frac{1}{3}  \uTheta^{1+ \frac{3}{4} } \ks_{ri}  \ks_{ls} \delta^{rs}\delta^{li}\delta_{jk}  -    \uTheta^{1+ \frac{3}{4} } \ks_{ji}  \ks_{lk}  \delta^{li}     + \frac{7}{4}  \delta^{rs}\delta^{li}\uTheta^{1+\frac{3}{2} } \ks_{jk}  \ks_{ri}  \ks_{ls}  \notag  \\
&     + \frac{1}{3}\delta^{rs}\delta^{li}\delta_{jk}   \uTheta^{3 -\frac{3}{4} } \kb_{ri}  \kb_{ls}  - \uTheta^{-1-\frac{3}{4}} \underline{\del{k}\del{j} \Phi } + \frac{7}{4}  \delta^{rs}\delta^{li}    \ks_{jk}\uTheta^{ 3 } \kb_{ri}  \kb_{ls} - \delta^{li} \uTheta^{3 -\frac{3}{4}} \kb_{ji}  \kb_{lk}\biggr] \label{e:dseq1}
\end{align}	
for $(t,\xi)\in[0,T)\times \del{}\Omega(0)$.
Then multiplying $2\ks^{jk}$ on both sides of \eqref{e:dseq1} (recall that we use $\delta^{jl}$ and $\delta_{jl}$ to raise and lower the index, for example $\ks^{jk}:=\delta^{jm}\delta^{kn}\ks_{mn}$),
\begin{align*}
\del{t}(\ks^{jk}\ks_{jk})=2\ks^{jk} \del{t}
\ks_{jk}	
=
& -\frac{1}{6}  \uTheta \ks_{jk} \ks^{jk}  + \biggl[  \frac{2}{3}   \uTheta^{1+ \frac{3}{4} }\ks^{jk} \ks_{ri}  \ks_{ls} \delta^{rs}\delta^{li}\delta_{jk}  -    2\ks^{jk}\uTheta^{1+ \frac{3}{4} } \ks_{ji}  \ks_{lk}  \delta^{li}    \notag  \\
&   + \frac{7}{2} \ks^{jk}  \delta^{rs}\delta^{li}\uTheta^{1+ \frac{3}{2} } \ks_{jk}  \ks_{ri}  \ks_{ls}    + \frac{2}{3} \ks^{jk}\delta^{rs}\delta^{li}\delta_{jk}   \uTheta^{3-\frac{3}{4} } \kb_{ri}  \kb_{ls}   \notag \\
& + \frac{7}{2}  \delta^{rs}\delta^{li}    \ks^{jk}\ks_{jk}\uTheta^{ 3 } \kb_{ri}  \kb_{ls} - 2 \delta^{li} \uTheta^{3 -\frac{3}{4}} \kb_{ji}  \kb_{lk} \ks^{jk} - 2 \ks^{jk}\uTheta^{-1-\frac{3}{4}}\underline{\del{k}\del{j} \Phi} \biggr]
\end{align*}
for $(t,\xi)\in[0,T)\times \del{}\Omega(0)$.
Then noting
\begin{equation}\label{e:negtm}
	-\int^t_0\frac{1}{6}  \uTheta(s,\xi) \kss(s,\xi) ds<0,
\end{equation}
with the help of \eqref{e:s0} and \eqref{e:Wrs4}, we arrive at
\begin{align}\label{e:dtS1}
&\del{t} \Bigl[\kss +\int^t_0\frac{1}{6}  \uTheta(s,\xi) \kss(s,\xi) ds \Bigr]
=
 \frac{7}{2}   \uTheta^{1+ \frac{3}{2} }   \kss^2    -    2\uTheta^{1+ \frac{3}{4} } \delta^{li} \ks_{ji}  \ks_{lk} \ks^{jk} \notag \\
&\hspace{0.5cm} + \frac{7}{2}\delta^{rs}\delta^{li}    \uTheta^{ 3} \kss \kb_{ri}  \kb_{ls}  - 2 \delta^{li} \uTheta^{3 -\frac{3}{4}} \kb_{ji}  \kb_{lk} \ks^{jk} - 2 \uTheta^{-1-\frac{3}{4}}\ks^{jk}\underline{\del{k}\del{j} \Phi}
\end{align}	
for $(t,\xi)\in[0,T)\times \del{}\Omega(0)$.

The next step is similarly integrating above \eqref{e:dtS1} to improve the estimate of $\kss$. However, before proceeding this, we calculate some quantities for later integration.

$(a)$ By Definition \ref{t:paramtr}.\eqref{t:Def1}, let us calculate
\begin{align}\label{e:xest0}
&\biggl(\frac{ 1}{\underline{\Theta}(t,\xi)(t+a) } \biggr)^{  3 } <  \biggl(\frac{ \frac{\lambda_0}{2\lambda_1}\sigma^{-1}+\frac{1}{3} t }{ t+\sigma^{-1}} \biggr)^{  3  }=\biggl(\frac{ \bigl(\frac{\lambda_0}{2\lambda_1}-\frac{1}{3}\bigr)\sigma^{-1} }{ t+\sigma^{-1}}+\frac{1}{3} \biggr)^{  3  } \leq  \biggl(\frac{\lambda_0}{2\lambda_1}\biggr)^{3 }
\end{align}

$(b)$ Due to $(\mathbf{w}_0|_{ \del{}\Omega(0)}, \Omega(0)) \in \mathfrak{A}^\dag(E,M,G_1,G_0,H^\prime(0),\sigma) \subset  \mathfrak{A}(E,M,G_1,H^\prime(0))$ for any $\sigma\in(0,\sigma_\dag)$ and by the proof of Theorem \ref{t:blupthm0}, we can apply Proposition \ref{t:spprbd} (recall Remark \ref{t:rmk1}) to conclude (see \eqref{e:xest}),
\begin{align*}
	|\chi(t,\xi) | > A (t+a)
\end{align*}
for $(t, \xi )\in[0,T)\times \del{}\Omega(0)$ where $T\leq T^\star$. Since $z_0(\xi) =  \lambda_1\lambda_0^{-1}\sigma  |\xi|$ (see \eqref{e:xidn0.a}) and $A:=\frac{1}{\lambda_1}z_0(\xi)$ (see \eqref{e:Az}), we arrive at
\begin{align*}
	|\chi(t,\xi)|>\lambda_0^{-1}\sigma  |\xi| (t+\sigma^{-1})
\end{align*}
for $(t, \xi )\in[0,T)\times \del{}\Omega(0)$.

$(c)$ Inserting above estimate to the following integration and noting \eqref{e:BinO.a}, \eqref{e:speint}, \eqref{e:xest0} yield
\begin{align}\label{e:intthex1}
	& \int^t_0  \frac{\uTheta^{ -\frac{7}{4}} }{|\chi(t,\xi)|^{3 }} ds=\int^t_0   \frac{\uTheta^{2 -\frac{3}{4}} }{(\uTheta |\chi(t,\xi)|)^{3 }} ds <\biggl(\frac{\lambda_0}{2\lambda_1}\biggr)^{3 } \frac{\lambda_0^3\sigma^{-3}}{|\xi|^3}\int^t_0 \uTheta^{2 -\frac{3}{4}}   ds \notag  \\
	& \hspace{1cm} < 12 \biggl(\frac{\lambda_0}{2\lambda_1}\biggr)^{3 } \frac{\lambda_0^3 }{|\xi|^3}   \Bigl(\frac{6\lambda_1}{\lambda_0}\Bigr)^{ \frac{1}{4}} \sigma^{-\frac{11}{4} }
 \leq      \frac{\lambda_0^3 }{6 G_1 } \biggl(\frac{6\lambda_1}{\lambda_0}\biggr)^{\frac{1}{4}}  \sigma^{ \frac{1}{4 } }  <  \frac{9 }{2 G_1 } \biggl(\frac{6\lambda_1}{\lambda_0}\biggr)^{\frac{1}{4}}  \sigma^{ \frac{1}{4 } }.
\end{align}

Now we are in the position to integrate \eqref{e:dtS1}.  With the help of $G_0=2G_1$,  \eqref{e:intthex1}, \eqref{e:negtm} and \eqref{e:Rest1}, it leads to
\begin{align*}
  \kss \leq & \kss +\int^t_0\frac{1}{6}  \uTheta(s,\xi) \kss(s,\xi) ds  \notag \\
\leq  &  \frac{1}{16} \sigma^{-1}  +
\frac{7}{2}  \sigma^{-2} \int^t_0 \uTheta^{ \frac{5}{2} }      ds  + 18   \sigma^{-\frac{3}{2 }} \int^t_0\uTheta^{ \frac{7}{4} }    ds  + \frac{63}{2}    \sigma^{-2}  \int^t_0\uTheta^{ 3 } ds\notag \\
&  + 18  \sigma^{-\frac{3}{2 }}  \int^t_0\uTheta^{\frac{9}{4}}   ds + 18 G_0 \sigma^{-\frac{1}{2 }} \int^t_0  \frac{\uTheta^{-\frac{7}{4}} }{|\chi(t,\xi)|^{3 }} ds  \notag \\
\leq  &  \frac{1}{16} \sigma^{-1}  +
7   \Bigl(\frac{6\lambda_1}{\lambda_0}\Bigr)^{ \frac{3}{2}} \sigma^{\frac{1}{2 } } \sigma^{-1}  + 72   \Bigl(\frac{6\lambda_1}{\lambda_0}\Bigr)^{\frac{3}{4}} \sigma^{\frac{1}{4 } }\sigma^{-1}+ \frac{189}{4}          \Bigl(\frac{6\lambda_1}{\lambda_0}\Bigr)^{2} \sigma  \sigma^{-1}   \notag \\
&  +   \frac{216}{5 }   \Bigl(\frac{6\lambda_1}{\lambda_0}\Bigr)^{\frac{5}{4}} \sigma^{\frac{3}{4 } } \sigma^{-1} +  \frac{81 G_0 }{ G_1 } \biggl(\frac{6\lambda_1}{\lambda_0}\biggr)^{\frac{1}{4}}  \sigma^{ \frac{3}{4}} \sigma^{ -1} <\frac{3}{16}\sigma^{-1}
\end{align*}
for $(t,\xi)\in[0,T)\times \del{}\Omega(0)$.

Then we turn to the improvement of the last quantity $|\kb_{jk}|$. Differentiating \eqref{e:Wrs3} and using \eqref{e:ray1}, \eqref{e:ray3}, \eqref{e:iWrs2b}--\eqref{e:iWrs3}, we derive that
\begin{align}\label{e:dtb1}
\del{t}\kb_{jk}
= & -2\uTheta^{-3} \uOmega_{jk}\del{t} \uTheta+\uTheta^{-2}\del{t}\uOmega_{jk}  \notag \\
= &  \frac{2}{3} \uTheta^{ - 1} \uOmega_{jk}  +  2 \uTheta^{-3} \uOmega_{jk}\delta^{rs}\delta^{li}\uXi_{ri}\uXi_{ls}    + 2 \uTheta^{-3} \uOmega_{jk} \delta^{rs}\delta^{li}\uOmega_{ri}\uOmega_{ls}  \notag  \\
& -\frac{2}{3}\uTheta^{ - 1} \uOmega_{jk} -    \delta^{li}\uTheta^{-2}\uXi_{jl}\uOmega_{ik} - \delta^{li}\uTheta^{-2}\uOmega_{jl}\uXi_{ik} \notag  \\
= & 2 \uTheta^{3} \delta^{rs}\delta^{li}  \kb_{jk}\kb_{ri} \kb_{ls}  + 2 \uTheta^{\frac{5}{2}}    \kb_{jk}\kss- \delta^{li}\uTheta^{ \frac{7}{4}} \ks_{jl} \kb_{ik}   - \delta^{li}  \uTheta^{ \frac{7}{4}} \ks_{ik} \kb_{jl}
\end{align}
for $(t,\xi)\in[0,T)\times \del{}\Omega(0)$.
Then integrating \eqref{e:dtb1} leads to
\begin{align}\label{e:dtb2}
\kb_{jk}=&\kb_{0jk}+ 2 \int^t_0  \uTheta^{3} \delta^{rs}\delta^{li}  \kb_{jk}\kb_{ri} \kb_{ls} ds + 2 \int^t_0 \uTheta^{ \frac{5}{2}}    \kb_{jk}\kss ds  - \int^t_0\delta^{li}\uTheta^{  \frac{7}{4}} \ks_{jl} \kb_{ik} ds  - \int^t_0\delta^{li}  \uTheta^{  \frac{7}{4}} \ks_{ik} \kb_{jl} ds.
\end{align}
Then estimating $|\kb_{jk}|$ by \eqref{e:dtb2}, with the help of \eqref{e:speint}, \eqref{e:Wrs5}, \eqref{e:dtdw2}, yields
\begin{align*}
|\kb_{jk}|\leq & |\kb_{0jk}| + 2 \int^t_0  \uTheta^{3}| \delta^{rs}\delta^{li}  \kb_{jk}\kb_{ri} \kb_{ls}| ds + 2 \int^t_0 \uTheta^{1+ \frac{3}{2}}  |  \kb_{jk}\kss| ds \notag  \\
& + \int^t_0 \uTheta^{1+\frac{3}{4}} |\delta^{li}\ks_{jl} \kb_{ik}| ds + \int^t_0 \uTheta^{1+\frac{3}{4}} |\delta^{li} \ks_{ik} \kb_{jl}| ds  \notag  \\
\leq & \frac{1}{4} \sigma^{-\frac{1}{2 }} + 18  \sigma^{-\frac{3 }{2 }} \int^t_0  \uTheta^{3}  ds + 2   \sigma^{-\frac{3 }{2 }}   \int^t_0 \uTheta^{1+ \frac{3}{2}}   ds + 6 \sigma^{-1} \int^t_0 \uTheta^{1+\frac{3}{4}}  ds   \notag  \\
\leq & \frac{1}{4} \sigma^{-\frac{1}{2 }} + 27   \Bigl(\frac{6\lambda_1}{\lambda_0}\Bigr)^{2} \sigma  \sigma^{-\frac{1}{2 }} +   4   \Bigl(\frac{6\lambda_1}{\lambda_0}\Bigr)^{ \frac{3}{2}}\sigma^{\frac{1}{2 }}  \sigma^{-\frac{1}{2 }} +  24  \Bigl(\frac{6\lambda_1}{\lambda_0}\Bigr)^{\frac{3}{4}}  \sigma^{\frac{1}{4 } }  \sigma^{-\frac{1}{2 }}  
< \frac{1}{2} \sigma^{-\frac{1}{2 }}
\end{align*}
for $(t,\xi)\in[0,T)\times \del{}\Omega(0)$.
Now we have improved the estimates of $|\ke|$, $ \kss $ and $|\kb_{jk}|$, then by bootstrap argument (see Appendix \ref{a:btpr}), we have proved \eqref{e:dtdw2} of the claim  for $(t,\xi)\in[0,T^\star)\times \del{}\Omega(0)$ and in turn, we complete the proof of this lemma.

\section{Discussions}\label{e:intsing}
The blowup theorems \ref{t:blupthm0} and  \ref{t:blupthm} imply  singularities. However, the Virial Theorem \ref{t:bluplem} can not give us more information on the possible singularities. Let us remark some interesting singularities. 

\subsection{Mass accretions (protostar formations) and fragmentation}
It is direct to verify 
\begin{equation}\label{e:starfrg}
	\int^{T^\star}_0  \underline{\Theta}(s,\xi)     ds=\pm\infty,
\end{equation}
for every  $\xi\in\mathring{\Omega}_\epsilon(0) $, is a special type of singularities given in the strong blowup theorem \ref{t:blupthm}, and if \eqref{e:starfrg} holds, by \cite[eq.$(3.13)$]{Liu2021d}, we reach $\underline{\rrho}(T^\star,\xi)=+\infty$ or $0$. $\underline{\rrho}(T^\star,\xi)=+\infty$ can be interpreted as the star formation (mass accretion), and $\underline{\rrho}(T^\star,\xi)=0$ implies the fragmentation process in astrophysics. Therefore, we need more information from data to conclude \eqref{e:starfrg} for the purpose of star formations and fragmentation.

\subsection{Physical vacuum boundary}
By \eqref{e:blup2} of the strong blowup theorem \ref{t:blupthm}, another possible singularity satisfies
	\begin{equation}\label{e:blup2aaaa}
	\int^{T^\star}_0  \|\nabla \rrho^{\frac{\gamma-1}{2}}(s)\|_{ L^\infty(\Omega(s))}   ds=+\infty
\end{equation}
which, in terms of the speed of sound, yields
	\begin{equation}\label{e:blup2aa}
	  \frac{1}{2\sqrt{K\gamma}} \int^{T^\star}_0  \| \rrho^{-\frac{\gamma-1}{2}}\nabla c_s^2(s)\|_{ L^\infty(\Omega(s))}   ds=\int^{T^\star}_0  \|\nabla c_s(s)\|_{ L^\infty(\Omega(s))}   ds=+\infty.
\end{equation}
We see the physical vacuum boundary is an example of such singularities since $\rrho\searrow 0$ and $\nabla c_s^2\in(0,\infty)$, by \eqref{e:blup2aa}, imply \eqref{e:blup2aaaa}.

\subsection{Shocks and shear/rotational singularities}
See \cite[\S VII]{Brauer1998} for  discussions of these cases.

\appendix

\section{Preliminary lemmas on  conservation laws}\label{a:cons}
\subsection{Proof of Lemma \ref{t:iden}}
\begin{proof}[Proof of Lemma \ref{t:iden}]
	Using Fubini–Tonelli  theorem\footnote{In this section, we can always use $\rrho \in C^1([0,\infty)\times \Rbb^3)$ and \eqref{e:spintg} to estimate and verify the condition of Fubini–Tonelli theorem, $\int_{\Omega(t)}\bigl(\int_{\Omega(t)}|f(\bx,\by)|d\bx\bigr) d\by<\infty$. We omit the details. } and the expression \eqref{e:dphi1} of $\del{j} \Phi$, the integral  $\int_{\Omega(t)}   \bigl( \rrho \del{ }^k \Phi  \bigr) d^3\bx$ can be expressed by
	\begin{align*}
	& \int_{\Omega(t)}  \rrho (t,\bx)\del{ }^k \Phi (t,\bx)  d^3\bx =  \int_{\Omega(t)} \rrho(t,\bx) \Bigl(\frac{1}{4\pi} \int_{\Omega(t)}    \rrho (t, \by) \frac{ x^k-y^k }{|\bx-\by |^3} d^3\by \Bigr) d^3\bx  \\
	&\hspace{0.5cm}=   -\int_{\Omega(t)} \rrho (t, \by) \Bigl( \frac{1}{4\pi} \int_{\Omega(t)}  \rrho(t,\bx)   \frac{y^k- x^k }{|\bx-\by |^3} d^3\bx\Bigr) d^3\by= -\int_{\Omega(t)}   \rrho (t,\by)\del{ }^k \Phi (t,\by)  d^3\by,
	\end{align*}
	that is
	\begin{equation*}
	\int_{\Omega(t)}  \rrho(t,\bx) \del{ }^k \Phi (t,\bx)  d^3\bx = -	\int_{\Omega(t)}  \rrho(t,\bx) \del{ }^k \Phi (t,\bx)   d^3\bx.
	\end{equation*}
	This concludes \eqref{e:rphin0}.
	
	Then we turn to \eqref{e:rpx3},
with the help of Fubini–Tonelli theorem, \eqref{e:Newpot}, \eqref{e:dphi1}, we derive
	\begin{align*}
	&\int_{\Omega(t)}  \rrho (t,\bx)x^j \del{j} \Phi(t,\bx)  d^3\bx =  \frac{1}{4\pi}\int_{\Omega(t)} \int_{\Omega(t)}   \rrho(t,\by)\rrho(t,\bx)\frac{\delta_{kj}(x^k-y^k) x^j }{|\bx-\by|^3} d^3\by d^3\bx \notag  \\
	& \hspace{0.5cm} =  \frac{1}{4\pi}\int_{\Omega(t)} \int_{\Omega(t)}   \rrho(t,\by)\rrho(t,\bx)\frac{\delta_{kj}(x^k-y^k)(x^j-y^j)}{|\bx-\by|^3} d^3\by d^3\bx \notag  \\
	& \hspace{1cm} +  \frac{1}{4\pi}\int_{\Omega(t)} \int_{\Omega(t)}   \rrho(t,\by)\rrho(t,\bx)\frac{\delta_{kj}(x^k-y^k) y^j }{|\bx-\by|^3} d^3\by d^3\bx  \notag  \\
	& \hspace{0.5cm} =  \int_{\Omega(t)}\rrho(t,\bx) \biggl(\frac{1}{4\pi}\int_{\Omega(t)}   \rrho(t,\by)\frac{1}{|\bx-\by| } d^3\by\biggr) d^3\bx \notag  \\
	& \hspace{1cm} -  \int_{\Omega(t)}\rrho(t,\by)y^j \biggl(\frac{1}{4\pi}\int_{\Omega(t)}   \rrho(t,\bx)\frac{\delta_{kj}(y^k-x^k)  }{|\bx-\by|^3} d^3\bx \biggr) d^3\by  \notag  \\
	&\hspace{0.5cm}=  -\int_{\Omega(t)}  \rrho(t,\bx)   \Phi(t,\bx) d^3\bx-  \int_{\Omega(t)}\rrho(t,\by)y^j \del{j}\Phi(t,\by) d^3\by .
	\end{align*}
By moving the last term of above identity to the left hand and changing its variable $\by$ to $\bx$, we arrive at
\begin{equation*}
	\int_{\Omega(t)}  \rrho(t,\bx)  \del{j} \Phi(t,\bx)  x^j d^3\bx = -\frac{1}{2}\int_{\Omega(t)}  \rrho (t,\bx)   \Phi(t,\bx)  d^3\bx.
\end{equation*}
	This completes the proof.
\end{proof}

\subsection{Proof of Lemma \ref{t:ctoms}}

\begin{proof}[Proof of Lemma \ref{t:ctoms}] The conservation of mass and Energy are fundamental and they can be obtained directly by the Euler equations, we omit the proofs and one can find them in various references. We only prove 
the conservation of velocities of center of mass in this appendix. This proof is due to direct calculations of $\frac{dx^k_c}{dt} $ and $\frac{d^2 x^k_c}{dt^2} $. We calculate the velocity of the center of mass $\frac{dx^k_c}{dt}$,
	\begin{align}\label{e:caldtxc1}
	\frac{dx^k_c}{dt} = & \frac{d }{dt} \Bigl(\frac{1}{ M }\int_{\Omega(t)} \rrho x^k d^3\bx \Bigr)  
	=  \frac{1}{ M } \int_{\Omega(t)} \del{t}  \rrho x^k d^3\bx  \notag  \\
	= & - \frac{1}{ M } \int_{\Omega(t)} \del{i} ( \rrho\mathring{w}^i) x^k d^3\bx  
	=   \frac{1}{ M } \int_{\Omega(t)}  \rrho\mathring{w}^k  d^3\bx,  
	\end{align}
	for $t\in [0,T)$.
	This concludes \eqref{e:vlcm}.
	
	Furthermore, we claim this velocity is a constant. To do so, we only need to prove $\frac{d^2 x^k_c}{dt^2}  \equiv 0$. Before that, we firstly derive an identity which is a sum of the Euler equations $\mathring{w}^k \times$ \eqref{e:NEul1}  and \eqref{e:NEul2},
	\begin{equation}\label{e:smeuler}
	\del{t}(\rrho\mathring{w}^k)+\del{i} (\rrho\mathring{w}^i\mathring{w}^k) +\delta^{ik} \del{i} p =- \rrho\del{ }^k  \Phi.
	\end{equation}
	Then by using \eqref{e:smeuler}, differentiating \eqref{e:caldtxc1} with respect to $t$ yields
	\begin{align*}
	\frac{d^2 x^k_c}{dt^2} = & \frac{d}{dt}\Bigl(\frac{1}{ M } \int_{\Omega(t)}  \rrho\mathring{w}^k  d^3\bx \Bigr)  
	=  \frac{1}{ M } \int_{\Omega(t)}  \del{t}\bigl(\rrho\mathring{w}^k \bigr) d^3\bx   \notag  \\
	= & \frac{1}{ M } \int_{\Omega(t)}   \Bigl(-\del{i}(\rrho\mathring{w}^i\mathring{w}^k)- \del{i}(\delta^{ik}  p)- \rrho \del{ }^k \Phi  \Bigr) d^3\bx   
	=   -\frac{1}{  M }  \int_{\Omega(t)}   \Bigl( \rrho \del{ }^k \Phi  \Bigr) d^3\bx \overset{\eqref{e:rphin0}}{=}0.
	\end{align*}

	This completes the proof of this lemma.
\end{proof}

\section{Tools of analysis}

\subsection{Bootstrap principle}\label{a:btpr}

The following bootstrap principle can be found in Tao's \cite[Proposition $1.21$]{Tao2006}. We present the statement of this principle without the proof. Please refer to \cite{Tao2006} for more detailed information and examples.
\begin{proposition}[Abstract Bootstrap Principle]\label{t:btpr}
	Let $\mathcal{I}$ be a time interval, and for each $t\in \mathcal{I}$ suppose we have two statements, a ``hypothesis'' $H(t)$ and a ``conclusion'' $C(t)$. Suppose we can verify the following four assertions:
	\begin{enumerate}
		\item (Hypothesis implies conclusion) If $H(t)$ is true for some time $t\in \mathcal{I}$, then $C(t)$ is also true for that time $t$. \item (Conclusion is stronger than hypothesis) If $C(t)$ is true for some time $t\in \mathcal{I}$, then $H(t^\prime)$ is true for all $t^\prime \in \mathcal{I}$ in a neighborhood of $t$.
		\item (Conclusion is closed) If $t_1$, $t_2, \cdots$ is a sequence of times in $\mathcal{I}$ which converges to another time $t\in I$, and $C(t_n)$ is true for all $t_n$, then $C(t)$ is true.
		\item (Base case) $H(t)$ is true for at least one time $t\in \mathcal{I}$.
	\end{enumerate}		
	Then $C(t)$ is true for all $t\in  \mathcal{I}$.
\end{proposition}

\subsection{Representation formulas for derivatives of Newtonian potential}\label{s:Nton}
The following two representation theorems for the derivatives of the Newtonian potential come from \cite[Lemma $4.1$ and $4.2$]{Gilbarg1983} and readers can found proofs therein.

First we denote $\Gamma$ is the fundamental solution of Laplace's equation given by
\begin{align*}
\Gamma(\bx-\by)=\Gamma(|\bx-\by|)=\begin{cases}
\frac{1}{n(2-n)\omega_n}|\bx-\by|^{2-n}, \quad & n>2  \\
\frac{1}{2\pi}\ln|\bx-\by|, & n=2.
\end{cases}
\end{align*}
where $\omega_n$ is the volume of unit ball in $\Rbb^n$. It is clear the following estimates of  derivatives hold
\begin{align*}
|\del{i}\Gamma(\bx-\by)|\leq \frac{1}{n\omega_n}|\bx-\by|^{1-n}  
\\
|\del{i}\del{j}\Gamma(\bx-\by)|\leq \frac{1}{\omega_n}|\bx-\by|^{-n}.  
\end{align*}

\begin{lemma}\label{t:dphi}
	Let $f$ be bounded and integrable in the bounded domain $D$ and let $\Phi$ be the Newtonian potential of $f$. Then $\Phi\in C^1(\Rbb^n)$ and for any $\bx\in D$,
	\begin{equation*}
	\del{i}\Phi(\bx)=\int_D \del{i}\Gamma(\bx-\by)f(\by)d^n\by
	\end{equation*}
	for $i=1,\cdots,n$.
\end{lemma}

\begin{lemma}\label{t:ddphi}
	Let $f$ be bounded and locally H\"older continuous with exponent $\alpha\leq 1$ in the bounded domain $D$ and let $\Phi$ be the Newtonian potential of $f$. Then $\Phi\in C^2(D)$, $\Delta \Phi=f$ in $D$, and for any $\bx\in D$,
	\begin{equation*}
	\del{i}\del{j}\Phi(\bx)=\int_{D_0} \del{i}\del{j}\Gamma(\bx-\by)(f(\by)-f(\bx))d^n\by-f(\bx)\int_{\del{}D_0}\del{i}\Gamma(\bx-\by)\nu_j(\by) d\sigma_\by
	\end{equation*}
	for $i,j=1,\cdots,n$, where $D_0$ is any domain  for which the divergence theorem holds and satisfies $D\subset\subset D_0$, and $f$ is extended to vanish outside $D$, and $\nu^j:=\delta^{ij}\nu_i$ is the outward-pointing unit normal vector of the boundary.
\end{lemma}

\subsection{Reynold's transport theorem}\label{s:RTT}
The following Reynold’s transport theorem is a multidimensional version of Leibniz integral rule, which states how to exchange the derivative and the integral if the region of integration is changing with time. The proof can be found in various references of fluid dynamics and calculus (see, for example, \cite[Page $578$]{Marsden2003} and \cite[\S$1$]{Chorin1993}).
\begin{theorem}[Reynold's Transport Theorem]\label{t:RTT}
	Suppose a field $C^1\ni f:[0,T)\times \overline{\Omega(t)} \rightarrow V$ and the flow $\chi:[0,T)\times \overline{\Omega(0)}\rightarrow \overline{\Omega(t)}\subset \Rbb^3$ generated by a vector field $w^i\in C^1([0,T)\times\overline{\Omega(t)},\Rbb^3)$, such that  $\chi(t,\xi)=\bx \in \overline{\Omega(t)}$ for every  $(t,\xi)\in[0,T)\times \overline{\Omega(0)}$ where $T>0$ is a constant, $\Omega(t):=\chi(t,\Omega(0)) \subset \Rbb^3$ is a domain depending on $t$ and $V \subset \Rbb^n$ for some $n \in \Zbb_{\geq 1}$. Then
	\begin{equation*}
	\frac{d}{dt}\int_{\Omega(t)}f(t,\bx)d^3\bx =  \int_{\Omega(t)}\Bigl[\del{t} f(t,\bx)  + \del{i} \bigl(f(t,\bx)w^i(t,\bx)\bigr) \Bigr]d^3\bx.
	\end{equation*}
\end{theorem}

\section*{Acknowledgement}
This work is partially supported by the China Postdoctoral Science Foundation Grant under the grant No. $2018M641054$ and the Fundamental Research Funds for the Central Universities, HUST: $5003011036$. Part of this work was completed during the program ``General Relativity, Geometry and Analysis: beyond the first 100 years after Einstein'' supported by the Swedish Research Council under grant no. 2016-06596 while the author was visiting Institut Mittag-Leffler in Djursholm, Sweden during September--December, 2019. I am very grateful for all the supports and hospitality during my visits.

\bibliographystyle{amsplain}
\bibliography{Reference_Chao}

\end{document}